\documentclass[paper=letter, fontsize=10pt,leqno]{scrartcl}	
\usepackage[T1]{fontenc}
\usepackage{fourier}

\usepackage[english]{babel}															
\usepackage{graphicx,amsmath,amsfonts,amsthm}										
\usepackage{url}

\usepackage{amssymb,bbm}
\usepackage[dvips]{color}
\usepackage{eucal}

\usepackage{sectsty}												
\allsectionsfont{\centering \normalfont\scshape}	

\usepackage{fancyhdr}
\newtheorem{theorem}{Theorem}
\newtheorem{assumption}{Assumption}
\newtheorem{lemma}{Lemma}
\newtheorem{proposition}{Proposition}

\newtheorem{definition}{Definition}

\newcommand{\norm}[1]{\left\| #1 \right\|}
\newcommand{\mc}[1]{\mathcal{#1}}
\newcommand{\mbb}[1]{\mathbb{#1}}

\DeclareMathOperator{\corr}{corr}

\numberwithin{equation}{section}		
\numberwithin{figure}{section}			
\numberwithin{table}{section}				

\title{
		\vspace{-1in} 	
		\usefont{OT1}{bch}{b}{n}
		\normalfont \normalsize \textsc{} \\ [25pt]
		\huge  Diffusivity in multiple scattering systems 
}

\author{\normalfont \large 
Timothy Chumley\footnote{Washington University, Department of Mathematics, Campus Box 1146, St. Louis, MO 63130}\, ,\  \ 
Renato Feres\footnotemark[1]\,  ,\ \  Hong-Kun Zhang\footnote{Department of Mathematics and Statistics, University of Massachusetts, Amherst, MA 01003, USA}
}

\date{\normalfont  \large \today}

\begin{document}

\maketitle
\vspace{-0.2in}
\begin{abstract}
\begin{center}Abstract\end{center}
We consider random flights of  point particles inside $n$-dimensional  channels of the form $ \mathbb{R}^{k}\times \mathbb{B}^{n-k}$,
where $\mathbb{B}^{n-k}$ is a ball of radius $r$ in dimension $n-k$. 
The particle  velocities immediately after each  collision with the boundary of the channel
comprise a Markov chain  with a  transition probabilities operator $P$ that 
  is determined by  a choice  of  (billiard-like) random mechanical   model of  the particle-surface    interaction at
  the ``microscopic'' scale.
 Our central concern is the relationship between the scattering  properties
encoded in  $P$   and the  constant of diffusivity
 of a   Brownian motion obtained by an appropriate limit of the random flight in the channel. 
  Markov operators obtained in this way  are {\em natural} (definition below), which means, 
 in particular, that (1) the   (at the surface) Maxwell-Boltzmann velocity distribution 
 with a given surface temperature,  when the 
  surface model contains moving parts, or (2)  the so-called Knudsen cosine law, when
 this model is purely geometric,
  is the  stationary   distribution of $P$.   We show by a suitable generalization of a central
 limit theorem of Kipnis and Varadhan  how the diffusivity is expressed in terms of the
 spectrum of $P$ and compute, in the case of $2$-dimensional channels, the exact values of the diffusivity for a class of
 parametric microscopic surface models of the above geometric type (2). 
  \end{abstract}

\section{Introduction and general definitions}
We consider mathematical models of particle-surface  systems involving multiple, or iterated,  (classical)  scattering.
The purpose of this section is to    motivate  the main   question  regarding these systems and to   introduce a few 
definitions needed  to
  informally   explain  our  results. More detailed statements
  will be given in the course of the paper.
\subsection{An idealized experiment and the main question}\label{idealized}
Figure \ref{exit flow} depicts an ideal experiment in which a small amount of gas composed 
of point-like, non-interacting  masses  is  injected into a (for simplicity of exposition $2$-dimensional) channel and the   amount of outflowing gas per unit time is recorded.  The graph on the right-hand side
shows a typical exit flow curve. Possible gas transport characteristics that can be 
obtained from such an experiment  are the mean value and higher moments  of the  molecular {\em  time of escape}.

The central question then is: what can these time characteristics of the gas outflow tell us about
the microscopic interaction (i.e., scattering properties) between gas molecules and the
surface of the  plates?

 For a more precise formulation of this question, we begin by describing 
  the classical surface scattering operators that model the
  microscopic  collisions of the point mass. We refer to the  boundary of the channel region as the (wall) {\em surface}, irrespective
  of its  actual dimension.

  Let $\mathbb{H}$ denote the upper-half plane, consisting of vectors  $(v_1,v_2)$
  with positive second component. Elements of  $\mathbb{H}$ represent velocities of a point mass immediately  after
  a collision with the surface. By identifying   $(v_1, -v_2)$ and $(v_1, v_2)$, we may
  regard pre-collision velocities as also being in $\mathbb{H}$. 
  A {\em collision event} is then specified by a measurable map $v\in \mathbb{H}\mapsto \eta_v\in \mathcal{P}(\mathbb{H}),$
 where $\mathcal{P}(\mathbb{H})$ indicates the space of probability measures on the upper-half plane.
 The measurability condition is understood as follows: For every essentially bounded  Borel measurable
 function $\phi$ on $\mathbb{H}$, the function 
 $$ v\mapsto (P\phi)(v):=\int_{\mathbb{H}} \phi(u)\, d\eta_v(u)$$
 is also measurable. We refer to $P$ as the {\em collision operator}. 
 This operator  specifies   the transition probabilities of  Markov chains with state space $\mathbb{H}$ giving 
the sequence  of post-collision velocities from which the molecular random flight inside the 
channel can be obtained. 
 
 \vspace{0.1in}
\begin{figure}[htbp]
\begin{center}
\includegraphics[width=5.3in]{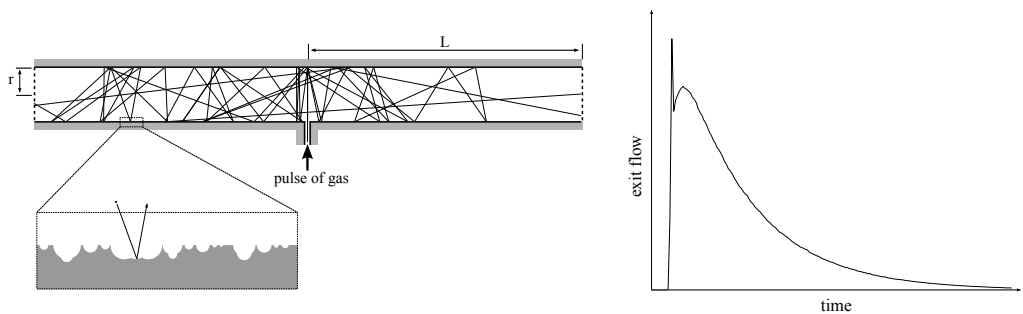}\ \ 
\caption{\small   Idealized experiment in which a small pulse of gas is injected into
a $2$-dimensional channel and the gas outflow is recorded. The graph on the right represents the rate at which gas escapes.  From this function it is possible to derive mean exit time of escape $\tau$.  The main problem is to
relate easily measured properties  of the gas outflow, such as $\tau$,  to the microscopic scattering characteristics
of the channel surface. }
\label{exit flow}
\end{center}
\end{figure}

If now  $\tau=\tau(L, r, s)$ denotes  the expected   exit time  of  the random flight,
where $s$ is  the molecular  root-mean square velocity, 
 then a more restricted form of the general question is to understand how $\tau$ depends on
 $P$.  (This expected exit time is easily measured in actual experiments involving gas diffusion using the rate of gas outflow as represented on the right hand side of Figure \ref{exit flow}; see, for example, \cite{marin} for
 so-called TAP-experiments in chemical kinetics.)
 
Although the analysis of $P$ is
 generally simpler  in dimension $2$,  an interesting complication arises here that is not present in 
 the case of a $3$-dimensional cylinder;  namely, with respect to 
 the stationary distribution of
 velocities for  natural collision  operators (see the next subsection), molecular displacement  between
 collisions has infinite variance  and   standard central limit theorems
 for Markov chains do
  not apply.  The same is true for the random flight in the region $\mathbb{R}^2\times [-r,r]$ between two parallel plates.
  In this regard, the random flight in   a $3$-dimensional cylindrical channel is simpler. 
(For an early study of  two parallel plates case, see  in \cite{bgt}.)
Infinite variance of the in-between collisions displacements  requires a generalization of the central 
 limit theorem of Kipnis and Varadhan in  \cite{kipnis} that  is  proved in  this paper.
  
 Now,  from an appropriate central limit theorem we obtain for $\tau$ the  asymptotic expression 
\begin{equation} \label{tau} \tau(L,r,s)\sim    \frac{L^2}{ \mathcal{D}\ln\left(\frac{L}{r}\right)} \end{equation}  
for long channels in dimension $2$,  i.e.,  for large values of   $L/r$, where $\mathcal{D}=\mathcal{D}(r,s)$ is the  diffusivity of 
a limit Brownian motion.  Therefore, a more specific formulation of the problem is to 
understand how properties of the collision operator are reflected on $\mathcal{D}$. 
A simple dimensional argument given in Subsection \ref{weak}  
shows that $$\mathcal{D}(r,s)=  \frac{4rs}{\pi}  \eta,$$ where $\eta$ only depends on the scattering 
characteristics at the microscopic scale  determined by $P$. The choice of constants will become clear shortly.

 The typical, but not the only  type of operator we consider here is defined by 
 a choice of microscopic contour of the channel wall surface, as suggested
 by Figure \ref{exit flow}. The main  problem then amounts to finding
the functional dependence of $\eta$ on geometric parameters of the surface
 microstructure. These parameters are scale invariant  and are typically 
 length ratios  and angles. The presence of the logarithmic term in $\tau$ is related to some surprising properties
 of $\mathcal{D}$,  as will be noted below, and 
for this  reason we give somewhat greater prominence   to the two-dimensional set-up in this paper.

\subsection{Natural collision  operators and microstructures}\label{natural}
Let $dV(v)$ denote the standard volume  element on $n$-dimensional half-space $\mathbb{H}:=\mathbb{H}^n$  and
define the probability measure   
$$ d\mu_{\beta}(v)=2\pi \left(\frac{\beta M}{2\pi}\right)^{\frac{n+1}{2}} \langle v, e_n\rangle  \exp\left(-\frac{\beta M}{2} |v|^2\right) \, dV(v).$$
on $\mathbb{H}$. We refer to $\mu_\beta$ as the 
  {\em surface Maxwellian}, or {\em surface Maxwell-Boltzmann distribution},   with parameter $\beta$  and   particle mass $M$.
Here, $\langle v, e_n\rangle$ denotes the standard inner product (dot product) of $v\in \mathbb{H}^n$ and 
the unit normal vector, $e_n$,  to the boundary surface at the origin.  We often denote this normal vector by $n=e_n$.
It will be  clear in context whether $n$ refers to dimension or to this normal vector. 

In physics textbooks, 
 $\beta=1/\kappa T$, where $T$ is 
absolute temperature and $\kappa$ is the {\em Boltzmann constant}.
A simple integral evaluation shows that the {\em mean squared post-collision speed} with respect to $\mu_\beta$ is
$$ s_{ms}^2:=\int_{\mathbb{H}^n} |v|^2 \, d\mu_\beta(v) =\frac{n+1}{\beta M}. $$

Another distribution of collision velocities that  arises naturally  
is concentrated on an hemisphere $S^+(s):=\{v\in \mathbb{H}: |v|=s\}$; it is 
defined 
by 
$$ d\mu(v) = \frac{\Gamma\left(\frac{n+1}{2}\right)}{s^n \pi^{\frac{n-1}{2}}} \langle v,n\rangle\, dV_{\text{\tiny sph}}(v)$$
where  $dV_{\text{\tiny sph}}(v)$ is the  volume element
on the hemisphere of radius $s$ induced from the ambient Euclidean space.  
In dimension $2$, 
$ d\mu(v)=\frac{1}{2 s^2} {\langle v, n\rangle}\, dS(v),$
where $dS$ indicates arclength element on $S^+(s)$.  
Equivalently,  $d\mu(\theta)=\frac12 \cos\theta d\theta,$ where 
  $\theta$ is  the angle between $v$ and the normal vector $n$.

 \begin{definition}[Natural collision operators]\label{nco} The collision operator $P$ will be called  {\em natural} if 
 one of the following holds:
(a)  $\mu_\beta$ is the unique stationary distribution for $P$, for some $\beta$; (b) 
the process defined by $P$ does not change the particle speed  and $\mu$
 is a stationary probability measure for $P$ for all $s$.  If case (a) holds     we say that the surface with
 associated operator $P$ has temperature $T=1/\kappa \beta$; in case (b)  we
 say that $P$ represents a {\em random reflection}.  In addition, we demand in both cases that
 $P$ and its stationary probability satisfy the detailed balance condition (see equation (20.5) of \cite{meyn}).  We use $\nu$ throughout the paper to indicate either $\mu_\beta$
 or $\mu$.
\end{definition}

\vspace{0.1in}
\begin{figure}[htbp]
\begin{center}
\includegraphics[width=3.3in]{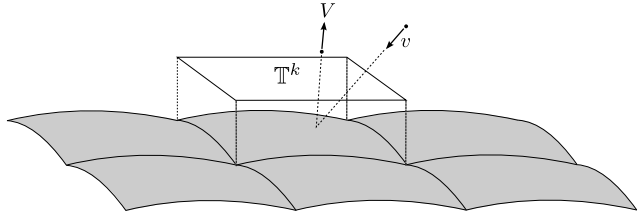}\ \ 
\caption{\small  A periodic microstructure without moving parts. The cube containing a period of the
microstructure defines a {\em cell}; the point   in a $k$-dimensional torus $ \mathbb{T}^k$ at which the  particle enters a typical  cell is 
assumed to be a uniform random variable. The particle enters with pre-collision velocity $v$ and exits with post-collision velocity $V$.  The corresponding $P$ describes a {\em random reflection}, as defined in the text.}
\label{wall}
\end{center}
\end{figure}

The natural operators of particular interest to us are those defined by a choice of 
surface microscopic structure. We briefly described them here. (See \cite{scott}, \cite{feres}, \cite{fz}, and \cite{fz2} for more details.)  By {\em surface (micro-)structure} we mean that the channel wall's surface has  a periodic relief composed of {\em cells}, each consisting of a mechanical system of moving masses or more simply a fixed geometric shape with no moving parts.  (Natural collision operators specified by the former correspond to part (a) of the definition, while the latter correspond to part (b).)  Moreover, the wall system, mechanical or purely geometric, is assumed to be at a ``microscopic'' length scale that, by definition,  is incommensurate with that of the channel  defined, say,  by the radius of the ball factor $\mathbb{B}^{n-k}$.  At a collision event, the point particle enters a  cell  of the wall system, undergoes one or more deterministic collisions with it,  transferring energy between wall and particle in case (a), and leaves with a seemingly\----from the perspective of the channel length scale\----random velocity $V$. (See Figure \ref{wall}.)  Because of this assumption of incommensurability  between the micro and macro scales,
 the relevant scattering properties specifying $P$ are invariant under
homotheties. Thus  in dimension $2$  we may, when convenient, assume 
  that the width of each   cell is    1.

  This incommensurability    also dictates our assumption that the particle  position on  the entrance of  a cell at the beginning 
  of a collision event  is a uniformly distributed random variable. When the wall system has moving parts, the kinetic state of the cell at the moment the particle enters a cell is
  also drawn from a fixed probability distribution (a canonical Gibbs state at temperature $T$). Under these assumptions,    
  $V$ is an actual  random variable and it can be shown  (see  \cite{scott})  that 
the associated collision   operator $P$   is natural according to Definition \ref{nco}. The   periodicity condition  is
not essential\----all the basic facts discussed here hold, for example,  for random  structures, defined as probabilistic  mixtures of
periodic  micro-structures. 

The operator $P$ can be expressed as follows.  Let $f$ be, say, a continuous bounded function on $\mathbb{H}$,
and let 
 $V_+$ be  the random velocity immediately after the collision of a particle with incoming velocity $V_-=v$. Then
 $$ (Pf)(v)=\mathbb{E}[f(V_+)|V_-=v].$$
For example, 
in the purely geometric case of   random reflections in dimension $2$,   $P$ is given by
\[
(Pf)(\theta) = \int_0^1 f\left(\Psi_\theta(r)\right) dr
\]
where $\Psi_\theta(r)$ is the angle $V$ makes with the normal vector $n$ and $r\in [0,1]$ is the
position
at which the particle enters a cell before collision.  A similar integral over $\mathbb{T}^k$ defines $P$ in general dimension.  

\vspace{0.1in}
\begin{figure}[htbp]
\begin{center}
\includegraphics[width=3in]{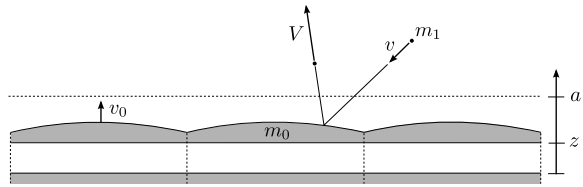}\ \ 
\caption{\small   An example of a wall system with moving parts. Mass $m_0$ can move freely up and down, bouncing 
off elastically against  the fixed floor and an upper limit that is permeable to $m_1$.}
\label{wall2}
\end{center}
\end{figure}

Figure \ref{wall2}   illustrates a  wall system with micro-structure having moving parts.  The periodic relief is assigned 
  a     mass $m_0$ and can  move vertically and 
  freely over  a short range of distances $[0,a]$ from the fixed  base, bouncing off elastically 
at the lower and upper limits. 
 The point particle of mass $m_1$ enters the wall system with velocity $v$ at a uniformly random location along $[0,1]$.  Upon entrance, the velocity of the wall 
is assumed to be normally distributed with mean $0$ and a given variance $\sigma^2$, where $m_0 \sigma^2$ is
proportional to the wall temperature; the particle then goes on to interact deterministically with the wall, leaving with random velocity $V$. 
This is a very special example of wall system for which 
  $P$ is natural. For other examples and   details about random billiards with microstructures  omitted here see \cite{scott}.  
See also \cite{jasmine} for a detailed analysis of  this and other examples in the context of stochastic processes in velocity space $\mathbb{H}$
having   stationary measure $\mu_\beta$ or $\mu$.

If we assume that the wall is static  and has  infinite mass,  the system of  the figure becomes  purely geometric  and the
particle mass plays no role. In this case the operator $P$ describes  a random reflection.
We refer to the subclass of natural operators derived from microscopic structures (either static or having moving
parts, with arbitrary surface contours) as operators {\em associated to surface microstructures.}  
An interesting question suggested by \cite{abs} is whether general natural operators are limits 
of operators associated to surface microstructures.  

\begin{proposition}
The operators   for the  classes of examples of Figures \ref{wall} and \ref{wall2} are natural. 
\end{proposition}
\begin{proof}
See \cite{scott,jasmine} for proofs of  these basic issues related to stationary measures    and more examples.
\end{proof}

Occasionally,     $\nu$ 
will stand for  either of the two measures $\mu_\beta$ or $\mu$ of Definition \ref{nco}. 
Because $P$ and $\nu$ are assumed to satisfy the detailed balance condition, $P$  
  is a self-adjoint  operator of norm $1$  on $L^2(\mathbb{H},\nu)$. A further assumption for the main results below  
is that $P$  be {\em quasi-compact}; that is,  the spectral radius  of $P$ restricted
to the orthogonal complement of the constant function in $L^2(\mathbb{H},\nu)$ is strictly less than $1$. 
Quasi-compactness for  natural collision operators for the types of  systems illustrated in Figure \ref{wall}
is known to hold in a number of cases. The static version of the system of Figure \ref{wall} (for the specific shape  shown in the figure)
 has this property  (see \cite{fz}, \cite{fz2}), and the operator for the one-dimensional version of the moving wall
is known to be compact (see \cite{scott});  the case of the two-dimensional moving wall    is still open.
Further examples of shapes of systems of the static type having quasi-compact $P$ will be provided later in this paper.

Noting  the two roles of $P$, as a self-adjoint operator on $L^2(\mathbb{H},\nu)$ and
as a Markov transition probabilities operator, a useful  characterization of  powers of $P$  
is as follows. Let $V_0, V_1, \dots$ be a stationary Markov chain with transitions $P$ and initial distribution $\nu$
and let $\Psi, \Phi\in L^2(\mathbb{H},\nu)$.
Then it  is not difficult to show  that
$$\left\langle \Psi, P^k\Phi\right\rangle=\mathbb{E}_\nu[\Psi(V_i)\Phi(V_{i+k})]$$
for any $i\geq 0$, 
where $\mathbb{E}_\nu$ indicates expectation given that $V_0$ is distributed according to $\nu$.

\section{Random flight and diffusivity}
\subsection{Between-collisions displacements and times}
The logarithmic term in Equation \ref{tau} is a special feature of the random billiard process in regions
bounded by parallel plates in arbitrary dimensions (in particular $2$-dimensional channels bounded by 
a pair of parallel lines), and it is not present in the more typical cylindrical channel region
$\mathbb{R}^k\times \mathbb{B}^{n-k}$ for $k=1,\dots, n-2$.   
Ultimately, this is  due to the mean square   displacements being infinite in the two-plates case and
finite in the other cases, as will be seen later.   This elementary but key observation is highlighted in the next proposition.

Let $\mathcal{C}:=\mathcal{C}^n:=\mathbb{R}^{k}\times \mathbb{B}^{n-k}$ denote the channel region.
Of special interest are the low dimensional cases: $n-k=1$, for $n=2,3$ (two-dimensional channels and slabs in dimension $3$) and
$k=1$, $n=3$ (cylindrical channels in dimension $3$).
  Let $\mathbb{H}_q$ represent the upper-half space  consisting of vectors  $v\in T_q\mathcal{C}$, $q\in \partial \mathcal{C}$, 
such that  $\langle n, v\rangle>0$, where $n$ is the unit vector in $\mathbb{H}_q$ perpendicular to 
$\partial \mathcal{C}$ and $\langle \cdot, \cdot \rangle$ is the inner  product  given by restriction of the standard  dot product in $\mathbb{R}^n$.
If $q'$ is the next collision point of the trajectory $t\mapsto q+t v$, let $Z(v)$ denote the natural  projection to the ``horizontal'' factor
$\mathbb{R}^k$ of the vector $q'-q\in \mathbb{R}^n$. We refer to $Z(v)$ as the {\em (horizontal) displacement vector} for the given  $v$.
See Figure \ref{cylinder}. The time of free flight between the collisions at $q$ and $q'$ will be indicated by $\tau_b(v)$.

\vspace{0.1in}
\begin{figure}[htbp]
\begin{center}
\includegraphics[width=2.5in]{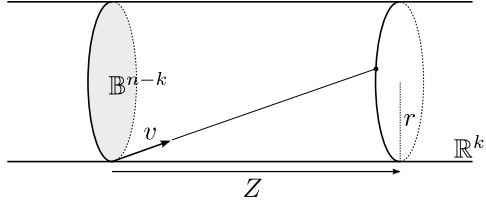}\ \ 
\caption{\small   The between-collisions displacement vector $Z(v)$, where $v$ is the post-collision velocity at a boundary point
of the channel region. The time between two consecutive  collisions  is denoted in this section $\tau_b(v)$.}
\label{cylinder}
\end{center}
\end{figure}

\begin{proposition}\label{volumes}
Let $\nu$ be either of the two probability measures of Definition \ref{nco}  (denoted there  $\mu_\beta$ and $\mu$). This 
is  a probability measure on $\mathbb{H}\cong T_q\mathcal{C}$ (concentrated on a hemisphere in the case of $\mu$), for a given collision point
 $q\in \partial \mathcal{C}$.   
Let $Z_a$ denote the product of $Z$   and   the indicator function 
 of the cone 
    $\mathbb{H}(a):=\{v\in \mathbb{H}: |Z(v)|\leq ar\}$ for $a>0$.  Also define for any unit vector $u\in \mathbb{R}^k$
    the orthogonal  projection $Z_a^u:=\langle u, Z_a\rangle$.
Then, if $n-k\geq 2$,  
 $$\mathbb{E}_\nu\left[\left(Z^u\right)^2\right] = \lim_{a \rightarrow \infty }\mathbb{E}_\nu\left[\left(Z^u_a\right)^2\right] =
  \frac{4r^2}{(n-k)^2-1}.$$
 If $n-k=1$, then the asymptotic expression 
 $$\mathbb{E}_\nu\left[\left(Z_a^u\right)^2\right] \sim
 {4r^2}  \ln a$$
holds.    The expected  time of free flight  $\tau_b$ (here `b' is for `between collisions') is finite for both types of measures and $n-k\geq 1$.  
For the stationary measure $\mu$ supported on the hemisphere of speed $s$, 
$$ \mathbb{E}_\mu\left[\tau_b\right]= \frac{2 r \sqrt{\pi}}{s (n-k)} \frac{\Gamma\left(\frac{n+1}{2}\right)}{\Gamma\left(\frac{n}{2}\right)}.$$
For $\mu_\beta$, the corresponding expression is
$$\mathbb{E}_{\mu_\beta}\left[\tau_b\right]=  \frac{r}{n-k}\sqrt{2 \pi \beta M}$$
where $M$ is particle mass.
\end{proposition}
For a sketch of the proof, see Section \ref{proofvolumes}.

Before considering random flights in $\mathcal{C}$,  we need to mention a    technical point of geometric interest that  only arises in dimensions $n>3$, having to do with
whether the operator $P$ actually gives rise to a well-defined process  in $\mathcal{C}$. The issue is
that in order for a fixed $P$ to induce  a scattering operator at each $T_q\mathcal{C}$, $q\in \partial \mathcal{C}$, we
need to be able to  identify the positive part (say, inward pointing) of this tangent space with the half-space $\mathbb{H}=\{v\in \mathbb{R}^n: \langle  v, e_n\rangle>0\}$. 
Such identification amounts to specifying an orthonormal  {\em frame} on the tangent space at each boundary point,
which provides the information of how the model microstructure is ``aligned'' with the channel wall. 
Introducing a frame field on $\partial \mathcal{C}$ has, however, the effect that the scattering operator $P_q$ at each 
boundary point $q$ becomes a conjugate of $P$ under an orthogonal transformation, rather than $P$ itself. 
For all the cases in dimensions $2$ and $3$, namely $(n=2, k=1), (n=3, k=1), (n=3, k=2)$, there is a natural frame (the parallel transported  frame over the boundary of the channel) with respect
to which $P_q=P$ for all $q$.  In the general case, $T_q\mathcal{C}$ has a canonical orthogonal decomposition
into a ``horizontal'' part, naturally identified with $\mathbb{R}^k$, and a vertical part that splits orthogonally into
the normal direction $\mathbb{R} n_q$ and the complement $\mathbb{V}_q$ of dimension $n-k-1$.
We can now understand the issue as follows: Let $v$ be a velocity vector of a particle that  emerges from a collision at
$q\in \partial \mathcal{C}$ and will collide again next at $q'$. In order that the Markov chain in velocity space
be given by iterates of the same operator $P$, we need a field of orthonormal frames with respect to which
$v$, at $q$, and its mirror reflection at $q'$  have the same representation as vectors in $\mathbb{R}^n$. 
The components of these two vectors in $\mathbb{V}^\perp_q$ and $\mathbb{V}^\perp_{q'}$, respectively, agree
if we choose the canonical (parallel) frame, but on the subspaces $\mathbb{V}_q$ themselves no such
frame  exists in general.  With this in mind, and to avoid complicating the picture by introducing such frame
fields as additional structure, we simply assume without further mention that $P$ is $\mathbb{V}$-{\em isotropic}, that is, 
it is invariant under conjugation by orthogonal linear maps that restrict to the identity   on $\mathbb{R}^k\oplus \mathbb{R} e_n$.
Notice that this assumption is vacuous in dimensions $2$ and $3$.

\subsection{Diffusivity, spectrum and mean exit time}\label{weak}
Let $\overline{X}_t$, $t\geq 0$,  be a piecewise linear path in the channel region $\mathcal{C}$ describing a random
flight governed by a natural  collision operator $P$.  Recall that $P$ is a self-adjoint operator on $L^2(\mathbb{H},\nu)$.
We assume throughout that $P$ is quasi-compact. (This is one of the conditions needed for  Theorem  \ref{clt}, below.) 

We wish to consider a diffusion process in $\mathbb{R}^k$ obtained by an appropriate scaling limit 
of  the  projection of $\overline{X}_t$ to the  $\mathbb{R}^k$ factor of $\mathcal{C}$. Let this projection be  denoted $X_t$, and assume  $X_0=0$. The sequence of post-collision velocities of $\overline{X}_t$ is a stationary Markov chain $V_0, V_1, \dots$, 
with initial distribution $\nu$.  The displacement vectors, previously  defined, are random variables $Z_0, Z_1, \dots$.
Thus $X_t$, at collision times, are sums of the $Z_i$. 
The {\em $a$-scaled} random flight is defined as follows. Let   $h(a)$, $a>0$, be  
$$h(a)=\begin{cases}a & \text{ for } n-k\geq 2\\   
a/\log a & \text{ for } n-k=1. \end{cases}
$$ 
Define the  {\em scaled channel system} with scale parameter $a>0$ 
to be the channel system with radius $r/a$ and 
  root-mean-square velocity  $h(a)s_{ms}$. (If $\nu=\mu$,   $s_{ms}$ is  the constant speed throughout the process and  if $\nu=\mu_\beta$,
  $s^2_{ms}=(n+1)/\beta M$.) 
  The random flight paths and their projection are defined as the paths for the $a$-system. We denote them
  by $\overline{X}_{a,t}$ and $X_{a,t}$, respectively. 
 The  $a$-scaled  free displacement  with post-collision velocity vector $v$ is $a^{-1}Z(v)$, and the displacements associated to the $V_j$ are 
    $a^{-1}Z_{j}$. For any $\tau > 0$, over a time interval $[0,\tau]$, the number of collisions of $\overline{X}_t$ with $\partial \mathcal{C}$
    will written  $N_\tau$.  For the $a$-scaled system  this number is $N_{a,\tau}:=N_{ah(a)\tau}$.  
        For the cases in which $n-k=1$, when $\mathbb{E}_{\nu}\left[|Z|^2\right]$ is infinite,  it will be  necessary to   also consider the  {\em $a$-truncation}  $Z_a$ of $Z$  introduced above in Proposition
    \ref{volumes}. The $a$-scaled $a$-truncation of $Z$ and $Z_j$ will be written $a^{-1}Z_a$ and $a^{-1}Z_{a,j}$.
Finally,  we will like 
to follow the 
projected  random flight  along an axis set by a unit vector  $u\in \mathbb{R}^k$. Thus we define
$Z^u$, $X^u_t$, $Z_a^u$, etc., to be the orthogonal projections of $Z$, $X_t$, $Z_a$, etc., on $\mathbb{R}u$. 

Theorem \ref{difflimitfiv} below gives conditions under which $X_{a,t}$ converges to Brownian motion for large $a$. 
In this subsection, we wish to focus on the variance (or diffusivity) of the limit Brownian motion,
and provide an interpretation of this constant in a way that does not make use of the physically somewhat artificial scaling
just introduced. 
    The precise conditions for the diffusion limit to exist when
    $Z$ has infinite variance relative  to $\nu$ (i.e., for $n-k=1$) 
    are not yet fully clear;  for Theorem \ref{difflimitfiv} we make the following  additional assumption 
    that  will be verified in the examples discussed later. 
    \begin{assumption}\label{assumption}
 Define for $\gamma>1$ the set $\mathbb{H}^\gamma(a):=\{v\in \mathbb{H}:|Z^u(v)|\leq ar/\log^\gamma a\}$ and 
 let $Z^u_{a,\gamma, j}$ be the product of $Z^u_j$ by the indicator function of\, $\mathbb{H}^\gamma(a)$. Then, for any $t >0$,  
 $$ \lim_{a\rightarrow \infty}  \mathbb{E}_\nu\left[\left(\frac1a\sum_{j=0}^{N_{a,t}} Z^u_{a,\gamma,j} \right)^2\right]$$
 exists for all unit vectors $u\in \mathbb{R}^k$.
    \end{assumption}
    This assumption, which will only be needed when $n-k=1$,  will be explained later  in the context of Theorems     \ref{clt} and 
\ref{wiptheor}. It is not needed for 
   $n-k\geq 2$, when $Z$ has  finite variance and 
      Theorem \ref{difflimitfiv} is for the most part  a consequence  of  well known   limit theorems, in particular the central limit theorem
      for reversible Markov chains as formulated by Kipnis and Varadhan in \cite{kipnis}. 
    
  \begin{theorem}[Diffusion limit]\label{difflimitfiv}
Let   $P$ be quasi-compact and, if $n-k=1$, suppose that  Assumption \ref{assumption} holds.   Then
  the $a$-scaled  projected random path $X^u_{a,t}$, for a unit vector $u\in \mathbb{R}^k$,
  converges weakly  as $a\rightarrow \infty$ to a Brownian motion in $\mathbb{R}u$  with diffusion constant $\mathcal{D}^u$ further
  specified below. 
  In particular,  $X^u_{a,t}$ converges   in distribution for each $t>0$ to a normal random variable in $\mathbb{R}u$ with
  mean $0$ and variance $t\mathcal{D}^u$. The following statements concerning $\mathcal{D}^u$ also hold:
  \begin{enumerate}
  \item   For the purpose  of having a baseline value for   $\mathcal{D}^u$, suppose  that  $P$ maps probability measures to $\nu$. In other words, let  the  velocity process on $\mathbb{H}$  be i.i.d. with probability measure $\nu$.  
 Let $n-k\geq 2$.  Then, denoting the constant $\mathcal{D}_0^u$ in this special case, 
  $$\mathcal{D}_0^u=\frac{4}{\sqrt{2\pi (n+1)}}\frac{n-k}{(n-k)^2-1} r s_{ms} $$
   when  $\nu=\mu_\beta$ and
   $$\mathcal{D}^u_0=\frac{2}{\sqrt{\pi}}\frac{\Gamma\left(\frac{n}{2}\right)}{\Gamma\left(\frac{n+1}{2}\right)}\frac{n-k}{(n-k)^2-1} rs $$
   when $\nu=\mu$. Recall that $s_{ms}$ is the root-mean-square velocity for $\mu_\beta$, and that $\mu$ is concentrated
   on the hemisphere of radius $s$ in $\mathbb{H}$.  Being independent of $u$, we denote these values by $\mathcal{D}_0$.
 \item For $n-k=1$,  the baseline diffusivities  
 are 
  $$\mathcal{D}_0 =\frac{4}{\sqrt{2\pi(n+1)}} rs_{ms}$$
    when $\nu=\mu_\beta$ and
    $$\mathcal{D}_0= \frac{2}{\sqrt{\pi}}\frac{\Gamma\left(\frac{n}{2}\right)}{\Gamma\left(\frac{n+1}{2}\right)}rs$$
    when $\nu=\mu$, supported on the hemisphere of  radius $s$.
   \item The diffusion constant for a general  $P$ can now be written as $\mathcal{D}^u=\eta(u) \mathcal{D}_0$, where $\eta(u)$ 
   has the following expression in terms of the spectrum of $P$.  
   First consider the case $n-k\geq 2$ and define a probability measure on the spectrum  by
    $$\Pi^u(d\lambda):=\|Z^u\|^{-2}\left\langle Z^u, \Pi(d\lambda) Z^u\right\rangle,$$ where
    $\Pi$ is  projection-valued   spectral measure associated to $P$ and the  inner product and norm are those of  $L^2(\mathbb{H},\nu)$. Then 
   \begin{equation}\label{eta}\eta(u)=\int_{-1}^1 \frac{1+\lambda}{1-\lambda}\, \Pi^u(d\lambda).\end{equation}
   Notice that $\eta(u)$ is quadratic in $u$.  Now suppose $n-k=1$ and define for each $a$ the probability   measure 
   $$\Pi^u_a(d\lambda):=\|Z^u_a\|^{-2}\left\langle Z^u_a, \Pi(d\lambda) Z^u_a\right\rangle$$ 
 on the spectrum and the function $\eta_a(u) = \int_{-1}^1 \frac{1+\lambda}{1-\lambda}\Pi_a^u(d\lambda)$. Then the limit   $\lim_{a\rightarrow \infty} \eta_a(u)$ exists and
defines a quadratic function $\eta(u)$ of $u$. 
  \end{enumerate}
    \end{theorem}  
    \begin{proof}
 The   limit theorems in probability theory we require  are standard in the case $n-k\geq 2$ and will be proved later for $n-k=1$
 in Theorems \ref{clt} and \ref{wiptheor}. Here  we only indicate how $\mathcal{D}_0$ and the expression for $\eta(u)$ in terms of the spectrum  of $P$ are obtained. 
    
    Recall that $N_{a,t}=N_{ah(a)t}$ is the number of collisions of the $a$-scaled random flight   with the
    boundary of the ($a$-scaled) channel region during time interval $[0,t]$. Let $\tau_j$ denote  
    the time duration of the   step   with (non-scaled) displacement $Z_{j}$.
    Then, as 
    $$({\tau_0+\cdots+\tau_{N_T-1}})/{N_T}\leq  {T}/{N_T}\leq({\tau_0+\cdots+\tau_{N_T}})/{N_T}$$
  for  any  $T>0$, we can apply
    Birkhoff's ergodic theorem 
to obtain 
    $\mathbb{E}_\nu[\tau_b] = \lim_{T\rightarrow \infty} {T}/{N_T},$
where we have used  a previous notation $\tau_b$ for the random time between consecutive collisions.
  
    Although not necessary in this case, we use here  the truncated displacement $Z_a$, so that  the derivation of the spectral
    formula will also apply to the infinite variance case to be discussed later (under the more stringent conditions needed
    in that  case). It will be shown later in Proposition \ref{prop51}
    that, for any $t>0$,  
   \begin{equation}\label{difflimit}\mathcal{D}^u=\lim_{a\rightarrow \infty}\frac{1}{a^2t} \mathbb{E}_\nu\left[\left(\sum_{j=0}^{N_{a,t}-1}Z^u_{a,j}\right)^2\right].\end{equation} 
    In the i.i.d. case,  
     this gives
    $$ \mathcal{D}^u=  \lim_{a\rightarrow \infty}\mathbb{E}_\nu\left[\left(Z^u_a\right)^2\right] \frac{N_{a,t}}{a^2t}=  \lim_{a\rightarrow \infty}\frac{h(a)}{a}\mathbb{E}_\nu\left[\left(Z^u_a\right)^2\right] \frac{N_{ah(a)t}}{ah(a)t}= \lim_{a\rightarrow \infty}\frac{h(a)}{a}\frac{\mathbb{E}_\nu\left[\left(Z^u_a\right)^2\right]}{\mathbb{E}_{\nu}[\tau_b]}.$$
We can now invoke Proposition \ref{volumes} to obtain the values claimed for $\mathcal{D}^u$ in the i.i.d. case.

Next we  obtain the spectral formula for $\eta(u)$, beginning from expression \ref{difflimit}. This expression 
 holds
without further assumptions in the finite variance case, and it follows from  Assumption \ref{assumption}
when $n-k=1$ as will be shown later in   Proposition \ref{prop51}. 
Because   $P$ has positive spectral gap and $Z^u_{a}$ has zero mean,  the measure $\Pi$ has compact support in the interval $(-1,1)$.
 In particular, $1-\lambda$ is bounded away  from zero on the support of   $\left\langle Z_a^u, \Pi(d\lambda) Z_a^u\right\rangle$.   
 Now observe that, for $j\geq i$, 
 $$\mathbb{E}_\nu\left[Z^u_{a,j} Z^u_{a,i}\right]= \left\langle Z^u_{a}, P^{j-i}Z^u_{a} \right\rangle= \int_{-1}^1 \lambda^{j-i}\, \|Z_a^u\|^2\Pi^u_a(d\lambda).$$
With this in mind, we obtain for a fixed $N$ after some algebraic manipulation,   
$$
 \mathbb{E}_\nu\left[\left(\sum_{j=0}^{N-1}Z^u_{a,j}\right)^2\right]=
 \int_{-1}^1 \left(N+ 2\sum_{j=1}^{N-1}\sum_{i=0}^{j-1} \lambda^{j-i}\right)\, \|Z_a^u\|^2\Pi^u_a(d\lambda)
 =\int_{-1}^1 \frac{1+\lambda}{1-\lambda} \left[N+O(1)\right] \, \|Z_a^u\|^2\Pi^u_a(d\lambda).
$$

The expectation on the right-hand side of limit \ref{difflimit} can be written as
 \begin{align*}
\frac{1}{a^2t}\sum_{N=1}^\infty \mathbb{E}_{\nu}\left[\left(\sum_{j=0}^{N-1}Z^u_{a,j}\right)^2\right]\mathbb{P}(N_{a,t}=N)
 &=\int_{-1}^1\frac{1+\lambda}{1-\lambda}{\mathbb{E}_{\nu}\left[\frac{N_{a,t}+O(1)}{ah(a)t}\right]} \frac{h(a)}{a}  \, \|Z_a^u\|^2 {\Pi^u_{{a}}(d\lambda)}.
   \end{align*}
   Keeping in mind the relationship between the expectation of $N_{a,t}$ and $\mathbb{E}_\nu[\tau_b]$ observed  above in the
   derivation of the i.i.d. case, we have
   $$\lim_{a\rightarrow \infty}\frac{1}{a^2t} \mathbb{E}_\nu\left[\left(\sum_{j=0}^{N_{a,t}-1}Z^u_{a,j}\right)^2\right]=  \lim_{a\rightarrow\infty} 
   \frac{h(a)}{a}  \frac{\mathbb{E}_\nu\left[\left(Z^u_a\right)^2\right]}{\mathbb{E}_\nu[\tau_b]}  
    \int_{-1}^1           \frac{1+\lambda}{1-\lambda} \,  {\Pi^u_a(d\lambda)}  = \mathcal{D}_0 \int_{-1}^1 \frac{1+\lambda}{1-\lambda}\, \Pi^u(d\lambda).$$
   This proves the claimed form of the diffusion constant. The necessary central limit theorem and weak invariance
   principle required to prove convergence to Brownian motion (in the case $n-k=1$) will be
   shown later.
          \end{proof}

We remark now on a simple interpretation of the diffusivity $\mathcal{D}$ in the context of the idealized experiment
described earlier that   allows us 
to obtain  $\mathcal{D}$ without  recourse to the somewhat physically  artificial $a$-scaling.
 In order  to keep  the discussion simple, only  the case $k=1$ is considered, although the main  idea  can be generalized 
in obvious ways.

Consider  a channel 
   $\mathcal{C}(L)=[-L,L]\times\mathbb{B}^{n-1}(r)$ of length $2L$ and 
recall that  $\tau(L, r, s)$   is the mean exit time from $\mathcal{C}(L)$, introduced in Subsection \ref{idealized},
  where $s$ is root-mean-square speed.
The following elementary dimensional properties are easily derived: 
\begin{enumerate}
\item[(i)] $\tau(L,r,s) = \tau(aL,ar,as)$
\item[(ii)] $\tau(L,r,s) = a\tau(L,r,as)$
\item[(iii)] $\tau(aL,r,s) = ah(a)\tau(L,r/a,h(a)s)$
\end{enumerate}
where the third property is a consequence of the first two. It also follows from  (i) and (ii) that the 
   function    $F(L/r):=(s/r)\tau(L, r, s)$  is independent of $s$,
  dimensionless (that is, devoid of physical units), and scale invariant.

  We are interested in the asymptotic behavior of $\tau(L,r,s)$ as $L$ grows to infinity.
Since the mean exit time  from the interval $[-r,r]$  for Brownian motion with diffusivity $\mathcal{D}$ starting
at $0$ is $r^2/\mathcal{D}$ we expect, given  Theorem \ref{difflimitfiv} and the above properties of the mean exit time,
\begin{equation}\label{asymptotics}\tau(L,r,s)\sim \begin{cases}  \frac{L^2}{\mathcal{D}}& \text{ if } n-k\geq 2\\
\frac{L^2}{\mathcal{D}\ln(L/r)} &\text{ if } n-k=1. 
\end{cases} \end{equation}
Notice, in particular, the expected relation $$\mathcal{D}= C(P) r s ,$$ where $C(P)=\lim_{a\rightarrow\infty}a h(a)/F(a)$, being independent of 
  $L, r, s$,   is  a  characteristic  number  of  the scattering process at a microscopic scale. 
  This asymptotic expression is indeed true, and it is a consequence of the following proposition, which
  will be proved later.
  \begin{proposition}\label{meanexittimelimit}
 Let $L>0$ and $\mathcal{T}$ be the function on the space of continuous paths   $\gamma:[0,\infty)\rightarrow \mathbb{R}$
 defined by $\mathcal{T}(\gamma):= \inf \{t\geq 0: |\gamma(t)|\geq L \},$ where the infimum of the empty set 
 is taken to be $-\infty$. Let $\mathbb{E}_0^a$ denote  expectation  with respect to the law of the process $t\mapsto X_{a,t}$,
 conditioned to start at $0$ and $\mathbb{E}_0^B$,  similarly defined,  for the Brownian motion with diffusion constant $\mathcal{D}$.
 Then $$\lim_{a\rightarrow\infty} \mathbb{E}^a_0[\mathcal{T}]=\mathbb{E}_0^B[\mathcal{T}]={\mathcal{D}}^{-1} {L^2}$$
 and the asymptotic expression \ref{asymptotics} holds. 
  \end{proposition}

\subsection{Generalized Maxwell-Smoluchowski models}
Before examining diffusivity in  examples of collision operators derived from geometric microstructures,
we briefly mention an interesting generalization of a very classical example widely used
in kinetic theory of gases, known as the  {\em Maxwell-Smoluchowski} 
collision   model.  The collision    operator for  this model is     $P_{\text{\tiny MS}}=\alpha Q_\beta+(1-\alpha) I$, where
  $\alpha\in [0,1]$ is a constant, $I$ is the identity operator, and $Q_\beta$ is 
  the projection into the subspace of constant functions: $$\left(Q_\beta\varphi\right)(v):=\int_{\mathbb{H}}\varphi (u)\, d\mu_\beta(u).$$
  The interpretation is that, upon collision with the surface, a particle scatters according to $\mu_\beta$ (diffuse scattering;
  see definition of $\mu_\beta$ at the beginning of Subsection \ref{natural})
  with probability $\alpha$, and reflects specularly with probability $1-\alpha$. 
  Clearly, 
   $P_{\text{\tiny MS}}$ is a natural collision operator according to Definition \ref{nco}. 
   A similar definition can be made for random reflection operators,  with $\mu$ in place of $\mu_\beta$. 
   In either case, the diffusivity is easily shown to be $\mathcal{D}=\frac{1+p}{1-p}\mathcal{D}_0$, where
   $p=1-\alpha$ is the probability of specular reflection and $\{1, p\}$ is the spectrum of $P_{MS}$.  This simple model is very  useful in providing 
   a rough interpretation of the typical $\eta$ obtained from micro-structures.

   A  generalization is as follows.
   Let $Q$ be a collision operator, not necessarily natural, such that for $v\in \mathbb{H}$ the measure
   $$A\mapsto Q(v,A) :=(Q\, \mathbb{1}_A)(v)$$
   is absolutely continuous with respect to the standard area measure, where $A$ is a measurable set and $\mathbb{1}_A$ denotes
  its  indicator function. We say that operators such as $Q$ are
   {\em diffuse}.  
    Let $\alpha:\mathbb{H}\times \mathbb{H}\rightarrow [0,1]$ be a measurable function.
Now form the {\em Metropolis-Hastings kernel} with proposal $Q$ and acceptance probability $\alpha$.
   The corresponding operator is 
   $$\left(P\varphi\right)(v)= \varphi(v) + \int_{\mathbb{H}} \alpha(v,u)\left[\varphi(u)-\varphi(v)\right]\, Q(v,du).$$
  The interpretation is that, for an incoming velocity $v$, a candidate outgoing velocity $V$ is chosen
  according to $Q(v, \cdot)$; then, with probability $\alpha(v,V)$,  $V$ is accepted as the post-collision velocity,
  and with probability $1-\alpha(v,V)$
the post-collision velocity is taken to be $v$ itself.

  The standard Metropolis-Hastings method provides an explicit   $\alpha$  for  a given $Q$ such
 that the resulting $P$ is natural with respect to $\mu_\beta$. See  \cite{tierney} for the general construction. This provides a large 
 class of examples of natural collision operators satisfying the following definition.
  
\begin{definition}[Generalized Maxwell-Smoluchowski models]
We say that the collision operator $P$ defines a generalized Maxwell-Smoluchowski model 
if it is a natural operator for a given $\mu_\beta$ and is of the  Metropolis-Hastings type with
a diffuse  proposal operator $Q$ and acceptance function $\alpha$. 
\end{definition}

Two families of examples of generalized Maxwell-Smoluchowski models are given in the next section.  In the family where a flat transition of  (scale-free) length $h$ is added between semicircles, the proposal $Q$ is simply the operator associated to the semicircle geometry and the acceptance function is the constant $1-h$.  In the family where the semicircle is split and a flat floor of length $h$ is added,   the acceptance function is more complicated as it depends on the incoming pre-collision angle.

Regarding these operators and their diffusivities, we only  indicate the following very general comparison result.  Let $P_1$ and $P_2$ be natural collision operators with the same stationary distribution $\mu$.  We say that  $P_1$ {\em dominates} $P_2$ {\em off the diagonal} if 
$
P_1\left(v,A\setminus\{v\}\right) \geq P_2\left(v,A\setminus\{v\}\right)
$
for $\mu$-almost all $v$.
The next proposition  holds with little modification to what is given in \cite{tierney}. 

\begin{proposition} Let $P_1$ and $P_2$ be natural collision operators  and let $\mc{D}_1$ and $\mc{D}_2$ be the associated diffusivities for the limit Brownian motion process in the scaled channel system.  If $P_1$ dominates $P_2$ off the diagonal then $\mc{D}_1 \leq \mc{D}_2$.
\end{proposition}

In other words, an operator that is more dispersing in the sense just defined has  slower diffusion. 
This result is illustrated  by the example of Figure  \ref{shapes2} of the next section featuring a flat transition between semicircles.   The formula for $\mathcal{D}(h)$ (next section) shows that the diffusivity increases with $h$, which is what is expected from the proposition.  A similar  comparison of operators does not hold for the example shown below where a flat floor is added.

\subsection{Examples of diffusivity for geometric microstructures}\label{examplesD}
We limit our attention in this subsection  to operators associated to static microstructures in dimension $2$. 
Therefore, the  main question of interest in how  the  shape of the surface contour 
    influences   the signature parameter $\eta=\mathcal{D}/\mathcal{D}_0$ of the particle-surface interaction. 
    For the  examples given here,
   $\eta$
   can be obtained  exactly. It will be noticed that the examples are variations on a theme: they
   are built out of arcs of circle and straight lines and correspond to focusing  billiards.

      \vspace{0.1in}
\begin{figure}[htbp]
\begin{center}
\includegraphics[width=3in]{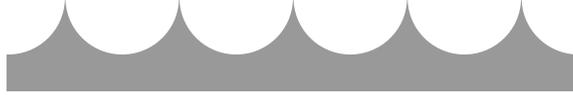}\ \ 
\caption{\small  A geometric microstructure  consisting of semicircular arcs. }
\label{shapes1}
\end{center}
\end{figure}

  For the focusing semicircle structure of   Figure
    \ref{shapes1}  it will be shown that 
\begin{equation}\label{diff1}
\eta =   \frac{1- \frac{1}{4}\log 3}{1+\frac{1}{4}\log 3}.
\end{equation}
This  value is  reminiscent of the limit variance seen in central limit theorems for the stadium billiard of deterministic billiard dynamics (see \cite{balint}).

A simple modification of the semicircles contour is shown in Figure \ref{shapes2}.
It consists of semicircles as in the first example separated  by flat sections.  We introduce the
parameter $h=l/(l+2r)\in (0,1)$, which gives the proportion of  the top line occupied by the flat part.

 \vspace{0.1in}
\begin{figure}[htbp]
\begin{center}
\includegraphics[width=4in]{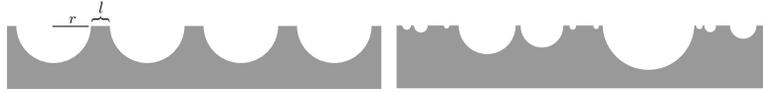}\ \ 
\caption{\small  Circles and flats.  For the non-periodic shape on the right,
we assume that the fraction of length comprising the flat part on top is well defined and equal to  $h=l/(l+2r)$, where $l$ and $r$ are
described on the left figure. The parameter $\eta$ is  the same in both cases.}
\label{shapes2}
\end{center}
\end{figure} 

  The diffusivity  for the shapes of Figure \ref{shapes2}, as a function of $h$, is given by
\begin{equation}\label{diff2}
\mc{D}(h) =\mathcal{D}_0 \frac{\eta+h}{1-h}
\end{equation}
where $\eta$ is the signature diffusivity  parameter  of the   example  of Figure \ref{shapes1}.  Clearly,  for the $h=0$ limiting case, the microscopic cell is simply the semicircle and $\mc{D}(0) = \mc{D}$.  At the other end, as $h$ approaches 1 the diffusivity increases without bound. Of course, for a completely flat surface, the transport ceases to be
a diffusion at all and becomes (a much faster) deterministic motion.

The next example refers to the surface of Figure \ref{shapes3}. In this case, the parameter $h$ measures the length
of the middle wall relative to the period length of the contour. 
 \vspace{0.1in}
\begin{figure}[htbp]
\begin{center}
\includegraphics[width=3in]{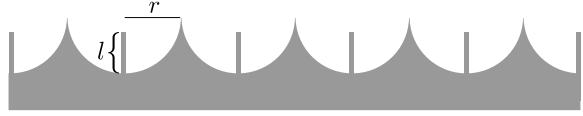}\ \ 
\caption{\small  Semicircles with middle wall. Define the scale free parameter $h=l/(2 r)$. }
\label{shapes3}
\end{center}
\end{figure} 

For the middle wall of relative  height $h < 1/2$ and $h=1/2$, the values are, respectively,
\begin{equation}\label{diff3}
\mc{D}(h) = \mathcal{D}_0 \frac{1-\frac{1}{4}\log 3}{1+\frac{1}{4}\log 3}, \quad \mc{D}(1/2) = \mathcal{D}_0 \frac{1+\frac{1}{4}\log 3}{1-\frac{1}{4}\log 3}.
\end{equation}
Observe, in particular, that 
  the diffusivity does not change, and has the same $\eta$ as  the example of Figure \ref{shapes1},
  until  the  middle walls   reach the top of the cell. At  that point    the diffusivity changes discontinuously
  to $\mathcal{D}(1/2)$. 
  
  A  related  phenomenon  is seen in the next   family of examples, shown in Figure \ref{shapes4}.
  It is obtained from the first example   by adding a flat  floor  of  relative length $h=l/(l+2r)\in [0,1)$.
It will be shown for this parametric family that 
\begin{equation}\label{diff4}
\mc{D}(h) =\mathcal{D}_0 \frac{1+\zeta_h}{1-\zeta_h},
\end{equation}
where
\[
\zeta_h = -\frac{1+3h}{4}\frac{1-h}{1+h}\log\frac{3+h}{1-h}.
\]

 \vspace{0.1in}
\begin{figure}[htbp]
\begin{center}
\includegraphics[width=3.0in]{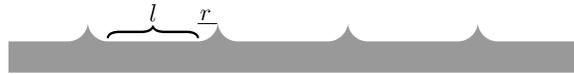}\ \ 
\caption{\small    A   family of  geometric microstructures with parameter
$h=l/(l+2r)$. }
\label{shapes4}
\end{center}
\end{figure}

At the $h=0$ limit  we naturally  have the same $\eta$ as for   the first example.  
What happens when  $h$ approaches $1$ is perhaps more surprising. In this case $\mathcal{D}(h)$ approaches the baseline
value $\mathcal{D}_0$. Recall that  this is the diffusivity of the process where at each  collision event  the particle reflects, independent of the pre-collision angle, according to the stationary measure $\mu$.  From the perspective of a single collision  event, collisions are nearly mirror-like; on the other hand, from a multiple scattering perspective the collision process reaches equilibrium instantaneously making the surface ideally rough in a sense.  This peculiar phenomenon and the discontinuity in $\mathcal{D}$ seen in the
previous example  are due to the fact that the diffusivity is determined only by collisions which occur at angles nearly parallel to the channel walls, a result made explicit in Proposition \ref{prop51}. An allusion to this property is found    in \cite{bgt}.

The computation of   $\mathcal{D}$ for these   examples will be given in Section \ref{examples}.

 \vspace{0.1in}
\begin{figure}[htbp]
\begin{center}
\includegraphics[width=3in]{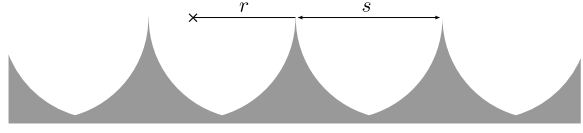}\ \ 
\caption{\small    The    family of  geometric microstructures of Figure \ref{shapes4} includes the above
when $l$ is allowed to be negative. Let $l=s-2r$,  where $r$ is the radius of the arcs of circle and $s$ is the width of a period cell
and $s$ is the length of the opening. For $l\geq -r$,  the diffusivity $\mathcal{D}(h)$, parametrized by
$h={l}/({l+2r})$, is still  given by  Equation \ref{diff4}. }
\label{shapes5}
\end{center}
\end{figure}

\section{The main limit theorems in the infinite variance case}\label{definitions}

Let $P$ be a natural collision operator with stationary measure $\nu$.   Define for each $v\in \mathbb{H}$
the measure
$
P(v,A) = \left(P\mbb{1}_A\right)(v)
$
for  $A \subset \mbb{H}$ measurable.  Let $\Omega = \mbb{H}^\mbb{N}$ and denote by $\mc{F}$   the product Borel $\sigma$-algebra. Define a measure $\mbb{P}$ on cylinder sets as
\[
\mbb{P}\left(\{\omega \in \Omega : \omega_0 \in A_0, \ldots, \omega_n \in A_n\}\right) = \int_{A_0}\int_{A_1}\cdots\int_{A_n} P(\omega_{n-1},d\omega_n) \cdots P(\omega_0, d\omega_1)\mu(d\omega_0),
\]
for $A_0, \ldots, A_n \subset \mbb{H}$ measurable and extend it  to    $\mathcal{F}$.  The coordinate projections $V_i : \Omega \to \mbb{H}$ given by $V_i(\omega) = \omega_i$, for $i=0, 1, \dots$,  thought of as random variables on the probability space $(\Omega, \mc{F}, \mbb{P})$, constitute  a Markov chain.  Because $\nu$ is stationary for $P$, it follows that the Markov chain   is stationary.  

Being a natural  operator,  $P$ together with $\nu$  satisfy   the detailed balance condition, which implies it is   self-adjoint  on $L^2(\mbb{H},\nu)$. We also assume from now on that $P$ is  quasi-compact, and so  has a {\em spectral gap}.
These assumptions  imply that the chain is $\rho$-mixing.  That is, 
\[
 \sup\left\{\corr(X,Y) : X \in L^2\left(\mc{F}_0^k\right), Y \in L^2\left(\mc{F}_{k+n}^\infty\right), k \geq 1\right\} =O\left(\rho^n\right),
\]
for some $\rho$ such that $0 < \rho < 1$, where the notations are as follows:
 $\corr(X,Y) $
is the correlation between $X$ and $Y$, $L^2\left(\mc{F}_0^k\right)$ is the space of square integrable
 functions measurable with respect to the $\sigma$-algebra  $\mathcal{F}_0^k$ generated by $V_0, \ldots, V_k$ and  
   $L^2\left(\mathcal{F}_{k+n}^\infty\right)$  is the space of square integrable functions 
  for  the $\sigma$-algebra
   generated by $V_j$, for $j\geq k+n$.

 Consider the $a$-scaled channel system, as defined earlier,  with channel radius scaled down by $a$ and each post-collision velocity scaled up by the previously defined $h(a)$.  
Let $X_{a,t}$ be the position at time $t$ along the horizontal axis of the particle for the $a$-scaled  random flight starting at $X_{a,0}=0$;
let $N_{a,t}$ be the number of collisions with the walls  for the scaled system  during   the time
interval $[0,t]$. Notice that  $N_{a,t} = N_{ah(a)t}$, where $N_s$ is, by definition,  the number of 
collisions during  $[0,s]$ for the non-scaled system. 

In what follows, $h(a)=a/\log a$ and  $Z$ will be any real measurable   function on $\mathbb{H}$   {\em slowly  varying} at infinity,
in the sense that 
$ \mathbb{E}_{\nu}\left[Z^2 \mathbb{1}_{|Z|\leq a}\right] \sim C \log a$
as $a\rightarrow \infty$, where $C>0$ is a constant and $a_n/b_n$ means $a_n/b_n\rightarrow 1$ as $n\rightarrow \infty$. 
The various notations used before for the inter-collision displacement function will be used for this general $Z$.
Thus, for example,   $Z_j = Z(V_j)$. However, the $a$-truncation $Z_a$ of $Z$ will be  understood   more
generally as follows: If  $I(a)$ is  an interval whose endpoints are  functions of the scaling parameter, we write
   $Z^{I(a)}:=Z\mathbb{1}_{\{Z\in I(a)\}}$, and if there is no ambiguity about  which interval is being assumed
   we write  $Z_a:=Z^{I(a)}$. Combining notations, $Z_{a,j}=Z_a(V_j)$,  for the non-scaled observable $Z$, while the
   corresponding  $a$-scaled quantity   (for the $a$-scaled system with radius $r/a$ and root-mean-square speed $h(a) s$) is $a^{-1}Z_{a,j}$.  Unless explicitly stated otherwise, $Z_a$ is associated to the interval $I(a)=[-a,a]$.
The probability measure $\Pi_a(d\lambda)$ on the spectrum on $P$ is defined as before for a general $Z$:
$\Pi_a(d\lambda):=\|Z_a\|^{-2}\left\langle Z_a, \Pi(d\lambda) Z_a \right\rangle. $
While we make no claims on the existence of a weak limit of the measures $\Pi_a$ we note that 
 there exists a subsequence of $\Pi_a$ that converges weakly to a probability measure which we will call $\Pi_0$.

 If a  sequence of random variables converges in distribution to
a normal random variable with mean $0$ and variance $\sigma^2$, we say for short  that the sequence converges
to  $\mathcal{N}(0,\sigma^2)$.

\begin{theorem}[Central limit theorem]\label{clt} Suppose $P$ is  a quasi-compact natural   operator  and that  Assumption \ref{assumption} holds. Then, for any $t>0$, 
the sum
$
\sum_{j=0}^{N_{a,t}-1} a^{-1}Z_{j}$ converges   to $ \mc{N}(0, t\mc{D})$
as $a\rightarrow \infty$, where $$\mathcal{D}=
\lim_{a\rightarrow \infty}\frac{1}{a^2t} \mathbb{E}_\nu\left[\left(\sum_{j=0}^{N_{a,t}-1}Z_{a,j}\right)^2\right].$$  
\end{theorem}

\begin{theorem}[Weak invariance principle]\label{wiptheor} Under the same assumptions as in Theorem \ref{difflimitfiv}, let $X_{a,t}\in \mathbb{R}^k$ be the particle at time $t$ in the  $a$-scaled system with radius  $r/a$ and root-mean-square velocity $h(a)s$.  Then $X_{a,t}$ converges weakly to $B_t$, a Brownian motion with diffusivity given by the  quadratic form $\mc{D}^u$.
\end{theorem}

We note that it is still possible, without Assumption \ref{assumption}, to prove a central limit theorem and weak invariance principle
for the inter-collision displacements\----only slight modifications of the statements and proofs are needed. 
However, Assumption \ref{assumption} allows us to express the variance of the limit distribution in the central limit theorem
in terms of the observable $Z$ and the operator $P$. 
 So while more general statements can be made and proven, we choose when possible to emphasize the connection between macroscopic data\----the limit variance in the central limit theorem\----and microscopic data encoded in  the operator $P$ and its spectrum.

We also note that Assumption \ref{assumption} may be reduced to a statement about covariances.  Observe that 
$$\mbb{E}\left[\left(\sum_{j=0}^{N_{a,t}-1} a^{-1} Z_{a,j}\right)^2\right] = \mbb{E}\left[a^{-2}\sum_{j=0}^{N_{a,t}-1} Z_{a,j}^2\right]+ \mbb{E}\left[a^{-2}\sum_{0 < |i-j| < N_{a,t}} Z_{a,i}Z_{a,j} \right].$$
As already seen in the proof of Theorem \ref{difflimitfiv} the limit of the first summand exists in general. Thus the assumption   may be restated as requiring the existence of the limit of the second summand of covariances.  In the case of strictly stationary $\rho$-mixing sequences with finite variance and a relatively light (less than exponential) condition on the rate of mixing, the scaled limit of variances is known to always exist (see for example \cite{bradley}).  As far as we are aware, the corresponding result in our setting\----whether such a limit always exists for exponentially fast $\rho$-mixing stationary sequences with infinite variance\----has not been addressed.

The remainder of the paper is organized as follows.
Section \ref{CLTsection}  outlines  the proof of the central limit theorem as a sequence of technical lemmas,
leaving the proofs of the  lemmas for Section \ref{lemproofs}. 
  In Section \ref{wip} we prove  the weak invariance principle and   Proposition \ref{meanexittimelimit} on convergence of mean exit times.  
   The computation of  diffusivity  for  the examples of Section
   \ref{examplesD}  
   is given in Section \ref{examples}.  The first subsection  there outlines a general technique for such computations, while the last two subsections are devoted to the computation of diffusivity for a periodic focusing semicircle micro-geometry  and
   related  parametric families.

\subsection{Outline of proof of the central limit theorem}\label{CLTsection}

In this section we explain  the skeleton of the proof of Theorem \ref{clt} as a sequence of lemmas, leaving the proofs of the  lemmas for Section \ref{lemproofs}.   
Now $Z$ represents  more generally an integrable scalar (rather than  vector)  random variable so that   $\mathbb{E}[|Z|]<\infty$, having mean $0$ and
 slowly varying at infinity; that is, $\mathbb{E}\left[(Z_a)^2\right]=O(\ln a)$.
Although $Z$ is more general than before,  we find it  convenient to 
 continue  to refer to $Z$ as the {\em displacement}. The typical random variable we wish to apply the below theorems
 to are the projections of the displacement vector $Z^u$ considered before. 
For a    channel in $\mathbb{R}^2$ bounded by parallel lines, we are mainly interested in $Z(v)=rv_1/v_2$, where
  $v=(v_1,v_2)$, with vector $(0,1)$ being perpendicular to the boundary lines.

\begin{lemma}\label{lem1} Let $\tau_b$ be the random inter-collision time. 
Then  $(ah(a)t)^{-1} N_{a,t}$ converges to $1/\mathbb{E}_\nu[\tau_b] $ almost surely  for each $t>0$ as $a\rightarrow \infty$.  In particular,  let 
$$
n_{a,t} := \left[ \frac{ah(a)t}{\mathbb{E}_\nu[\tau_b]} \right],
$$
where $[x]$ denotes the integer part of  $x$. Then  
$
{N_{a,t}}/{n_{a,t}} \to 1
$
almost surely.
\end{lemma}

The  proof of Lemma \ref{lem1} has already been given in   the  proof of  Theorem \ref{difflimitfiv}   in Section
\ref{weak}.

\begin{lemma}\label{lem2} Let $n_{a,t}$ be defined as in Lemma \ref{lem1}.  If  $\,\sum_{j=0}^{n_{a,t}-1} a^{-1}Z_{j}$ converges
in distribution to $ \mc{N}(0,t\mc{D})$ then so does 
$
\sum_{j=0}^{N_{a,t}-1} a^{-1}Z_{j}
$.
\end{lemma}

The next lemma shows that one can work with the $a$-truncated random variables $Z_{a,j}$ instead of the $Z_j$.
\begin{lemma}\label{lem3}
The quantity 
$\sum_{j=0}^{n_{a,t}-1} a^{-1}\left(Z_{j} - {Z}_{a,j}\right)$ converges to $0$ in probability as $a\rightarrow \infty$. 
\end{lemma}

To address the issue of statistical dependence among the displacements, we employ Bernstein's big-small block technique.  That is, we break the sum of truncated displacements into alternating big and small blocks in such a way that the small blocks are negligible and the big blocks are in a sense independent.
Let $\alpha = 0.01, \beta = 0.6$ and define
$
b_{a,t} = \left[n_{a,t}^\beta\right]$ and $ s_{a,t} = \left[n_{a,t}^\alpha\right].
$
These are 
the lengths of the big blocks and small blocks, respectively.  Define the big blocks as
$$
U_{a,i} = \sum_{j=1}^{b_{a,t}} a^{-1}{Z}_{a,(i-1)(b_{a,t}+s_{a,t})+j},
$$
for $1 \leq i \leq k_{a,t}$, where $k_{a,t}$ is the largest integer $i$ for which $(i-1)(b_{a,t}+s_{a,t})+b_{a,t} < n_{a,t}-1$.  Note  that $k_{a,t} \sim n_{a,t}^{1-\beta}$. 
 Next define the small blocks $V_{a,i}$ as the sums that remain between the big blocks.  That is,
$$
V_{a,i} =\sum_{j=1}^{s_{a,t}}a^{-1} {Z}_{a,(i-1)(b_{a,t}+s_{a,t})+b_{a,t}+j},
$$
for $1 \leq i < k_{a,t}$, and
$
V_{a,k_{a,t}} = a^{-1}\left({Z}_{a,(k_{a,t}-1)(b_{a,t}+s_{a,t})+b_{a,t}+1} + \cdots + {Z}_{a,n_{a,t}-1}\right),
$
so that 
$$
\sum_{j=0}^{n_{a,t}-1} a^{-1}{Z}_{a,j} = \sum_{i=1}^{k_{a,t}}\left( U_{a,i}+V_{a,i}\right).
$$
The next lemma shows that it will suffice to consider only the   big blocks as the sum of the small blocks is negligible in probability.
\begin{lemma}\label{lem4}
The sum
$
\sum_{i=1}^{k_{a,t}} V_{a,i}
$
converges to $0$ in probability  as $a\rightarrow \infty$. 
\end{lemma}

The proof of Theorem \ref{clt} has by now been reduced to showing that 
$
\sum_{i=1}^{k_{a,t}} U_{a,i} 
\rightarrow  \mc{N}(0,t\mc{D})$  as $a\rightarrow \infty$.
Therefore, the theorem will be proved if we  show that the characteristic function of this  sum   converges to the characteristic function  of the normal random variable.
By the next lemma, the big blocks are asymptotically independent in the sense that the characteristic function of their sum can be estimated by the product of their characteristic functions.

\begin{lemma}\label{lem5}Convergence 
$
\left|\mathbb{E}\left[\exp\left(i\mu \sum_{i=1}^{k_{a,t}} U_{a,i}\right)\right] -\prod_{i=1}^{k_{a,t}}\mathbb{E}\left[\exp\left(i\mu\, U_{a,i}\right)\right]\right| \to 0
$
holds for all  $\mu$ in $\mbb{R}$ as $a \to \infty$.
\end{lemma}
Combining the above with the following lemma then gives the proof of Theorem \ref{clt}.
\begin{lemma}\label{lem6}The   convergence 
$
\prod_{i=1}^{k_{a,t}}\mathbb{E}\left[\exp\left(is\, U_{a,i}\right)\right] \to \exp\left(-\frac{s^2}{2}t\mc{D}\right)
$
holds for all $s \in \mbb{R}$ as $a \to \infty$.
\end{lemma}

\subsection{Proof of the weak invariance principle}\label{wip}

We give now a proof of Theorem \ref{wiptheor} 
  and of Proposition \ref{meanexittimelimit}.  To show weak convergence in the space $C[0,\infty):=C([0,\infty),\mathbb{R}^k)$
  of continuous paths in $\mathbb{R}^k$  it suffices to show that the finite dimensional distributions of the projection of $X_a$
  along each coordinate axis of $\mathbb{R}^k$  converge weakly to those of Brownian motion $B$ and the collection $X_a$ in $C[0,\infty)$ is tight (see, for example, \cite{bill}).  In fact, with regard to the second condition it suffices to show tightness of the collection restricted to $C[0,t]$ for all $t > 0$ (see, for example, \cite{ww}).  Thus the following two propositions are sufficient to prove the theorem. We assume
  without loss of generality that dimension $k=1$. 

\begin{proposition} Under the conditions of Theorem \ref{wiptheor}, the  random vectors $(X_{a,t_1}, \ldots, X_{a,t_l})$ converge
 weakly to $(B_{t_1}, \ldots, B_{t_l})$ as $a\rightarrow \infty$ for all $l$ and all $t_1< \cdots < t_l \in [0,\infty)$.
\end{proposition}
\begin{proof}
The proof is by induction on $k$.  For the case $k=1$ we begin by writing
$$
X_{a,t_1} = \left(\sum_{j=0}^{N_{a,t_1}-1} a^{-1}Z_{j}\right) + R_{a,t_1},
$$
where $R_{a,t_1}$ is the signed distance traveled in the time between collision $N_{a,t_1}$ and $t_1$.  We claim that $R_{a,t_1} \to 0$ in probability as $a\rightarrow \infty$.  Indeed   note that $|R_{a,t_1}|\leq |a^{-1}Z_{N_{a,t_1}}|$ so, for any $\epsilon > 0$
$$
\mbb{P}\left(\left|R_{a,t_1}\right| > \epsilon\right)
\leq \mbb{P}\left(\left|a^{-1}Z_{N_{a,t_1}}\right| > \epsilon\right)
 \leq \frac{\mbb{E}[\left|Z\right|] }{\epsilon a}, 
$$
which goes to $0$ as $a\rightarrow\infty$. 
It follows   that $X_{a,t_1}$ converges in distribution to $B_{t_1}$ by Theorem \ref{clt}.

Next we consider the case $l>1$.  It suffices to show that 
$
\sum_{i=1}^l \xi_i X_{a,t_i}$ converges in distribution to $ \sum_{i=1}^l \xi_i B_{t_i}
$ as $a\rightarrow \infty$
for any $(\xi_1, \ldots, \xi_l) \in \mbb{R}^l$.  We first write
\[
\sum_{i=1}^l \xi_i X_{a,t_i} = \xi_1X_{a,t_1} + \cdots + \xi_{l-2}X_{a,t_{l-2}} + \left(\xi_{l-1}+\xi_l\right)X_{a,t_{l-1}} + \xi_l\left(X_{a,t_l}-X_{a,t_{l-1}}\right)
\]
Arguing as in the previous case and using the techniques of truncation and Bernstein's method as in the proof of the central limit theorem we conclude that the first $l-1$ summands above are asymptotically independent from the last so that, by the induction hypothesis, the sum
$$
\xi_1X_{a,t_1} + \cdots + \xi_{l-2}X_{a,t_{l-2}} + \left(\xi_{l-1}+\xi_k\right)X_{a,t_{l-1}} + \xi_l\left(X_{a,t_l}-X_{a,t_{l-1}}\right)$$
converges in distribution as $a\rightarrow \infty$ to 
$$\xi_1B_{t_1} + \cdots + \xi_{l-2}B_{t_{l-2}} + \left(\xi_{l-1}+\xi_l\right)B_{t_{l-1}} + \xi_lB_{t_l-t_{l-1}} = \sum_{i=1}^l \xi_i B_{t_i}.
$$
This concludes the proof of the proposition.
\end{proof}

\begin{proposition} Let $t > 0$
and  define    
 $u(a,\delta):= \sup\left\{
\left|X_{a,u} - X_{a,v}\right|:  |u-v| < \delta \text{ and } u,v\in[0,t]\right\}.$
Then 
$
\lim_{\delta\rightarrow 0}\lim_{a\rightarrow \infty}\mbb{P}\left(u(a,\delta) > \epsilon
 \right) =0
$  for all $\epsilon> 0$.
\end{proposition}
\begin{proof}
Let $\epsilon > 0$ and $\delta < t$.  
For simplicity we assume $\delta$ divides $t$ and let $n = t/\delta$.  The argument holds in general with only minor modification.  Let
$
0 = t_0 < \cdots < t_n = t
$
be the equidistant partition of $[0,t]$.  Observe that
$$
\mbb{P}\left(u(a,\delta)> \epsilon \right) \leq \mbb{P}\left(\max_{0 \leq j \leq n-1}\sup_{s \in [t_j,t_{j+1}]} \left|X_{a,s} - X_{a,t_j}\right| > \epsilon/3 \right) \ 
 \leq \sum_{j=0}^{n-1} \mbb{P}\left(\sup_{s \in [t_j,t_{j+1}]} \left|X_{a,s} - X_{a,t_j}\right| > \epsilon/3 \right).
$$
Introducing the notation 
$D_k^l:=\sum_{j=k}^{l-1}Z_{a,j}-\sum_{j=k}^{N_{a,t}-1}Z_{a,t},$
then for any $0 \leq j \leq n-1$, the event
$
\sup\left\{ \left|X_{a,s} - X_{a,t_j}\right|:    s \in \left[t_j,t_{j+1}\right]       \right\}> \epsilon/3
$
implies  that 
$
\max_{N_{a,t_j} \leq k \leq N_{a,t_{j+1}+1}} \left|
D_0^k\right| > a\epsilon/6.
$
Next let $n_\delta = \left[\frac{ah(a)\delta}{v/r\pi}\right]$ and let $A(j,\delta)$ denote the event   $N_{a,t_{j+1}}-N_{a,t_j} + 1 \leq n_\delta$.  Then the probability $\mbb{P}\left(\sup_{s \in [t_j,t_{j+1}]} \left|X_{a,s} - X_{a,t_j}\right| > \epsilon/3, A(j,\delta)\right)$ is bounded above
by 
\begin{align*}
  \mbb{P}\left(\max_{N_{a,t_j} \leq k \leq N_{a,t_{j+1}+1}} \left|D_0^k\right| > a\epsilon/6, A(j,\delta)\right) 
&\leq \mbb{P}\left(\max_{N_{a,t_j} \leq k \leq N_{a,t_j}+n_\delta} \left|D_0^k\right| > a\epsilon/6\right) \\
&\leq 2\mbb{P}\left(\max_{1\leq k\leq n_\delta} \left|\sum_{i=0}^{k-1}Z_{i}\right| > a\epsilon/6\right) 
\leq 4\mbb{P}\left(\left|\sum_{i=0}^{n_\delta-1}Z_{i}\right| > a\epsilon/6\right),
\end{align*}
where the last two inequalities follow as in the proof of Lemma \ref{lem2}.
Therefore, 
\begin{align*}
\mbb{P}\left(u(a,\delta)> \epsilon \right)&\leq \sum_{j=0}^{n-1} \mbb{P}\left(\sup_{s \in [t_j,t_{j+1}]} \left|X_{a,s} - X_{a,t_j}\right| > \epsilon/3 \right) 
\leq \sum_{j=0}^{n-1} \left(4\mbb{P}\left(\left|\sum_{i=0}^{n_\delta-1}Z_{i}\right| > a\epsilon/6\right) + \mbb{P}\left(A(j,\delta)^c\right)\right)\\
&= n \left(4\mbb{P}\left(\left|\sum_{i=0}^{n_\delta-1}Z_{i}\right| > a\epsilon/6\right) + \mbb{P}\left(A(j,\delta)^c\right)\right)
= \frac{t}{\delta} \left(4\mbb{P}\left(\left|\sum_{i=0}^{n_\delta-1}Z_{i}\right| > a\epsilon/6\right) + \mbb{P}\left(A(j,\delta)^c\right)\right).
\end{align*}
Notice that the number of collisions in the interval $[t_j,t_{j+1}]$ is precisely $N_{a,t_{j+1}}-N_{a,t_j}+1$ and so by Lemma \ref{lem1}, 
$
 \frac1{n_{\delta}}\left({N_{a,t_{j+1}}-N_{a,t_j}+1}\right) \to 1
$
almost surely as $a\to\infty$.  From this it follows that
$
\mbb{P}\left(A(j,\delta)^c\right) \to 0
$
almost surely as $a\to\infty$, independent of $j$ and $\delta$.  From Theorem \ref{clt}
\[
\frac{1}{\delta}\mbb{P}\left(\left|\sum_{i=0}^{n_\delta-1}a^{-1}Z_{i}\right| > \epsilon/6\right) \to \frac{2}{\delta^{3/2}\sqrt{2\pi\mc{D}}}\int_{\epsilon/6}^\infty e^{-x^2/2\delta\mc{D}}\,dx
\]
as $a\to\infty$.  Letting $\delta \to 0$ on the right-hand side above then gives the result.
\end{proof}

\begin{proof}[Proof of Proposition \ref{meanexittimelimit}]
The equality in the statement of theproposition is a standard fact on the mean exit time of Brownian motion from an interval.  We prove here only the convergence of mean exit times.
Note that by the continuous mapping theorem
$
\tau(X_a)$ converges in distribution to $ \tau(B).$
(Of course $\tau$ is not continuous on all of $C[0,\infty)$ but it's not difficult to show that it is $\mbb{P}_0^B$-a.s. continuous.)  Therefore, to show the convergence of mean exit times it suffices to show that the collection of $\tau(X_a)$ is uniformly integrable.  That is, it suffices to show that for any $\epsilon > 0$ there exists $M > 0$ such that
$
\mbb{E}_0^a\left[\tau\mbb{1}_{\{\tau> M\}}\right] < \epsilon
$
for all $a$.

Note that for all $\epsilon > 0$ there exists $\delta \in (0,1)$ such that
$
\mbb{P}_x^B(\tau < \epsilon) > \delta
$
for all $x \in (-L,L)$.  Since $\tau(X_a) $ converges to $\tau(B)$ in distribution it follows, similarly,  that   for any $\epsilon>0$ there exists some $\delta \in (0,1)$ and $a_0$ such that 
$
\mbb{P}_x^a(\tau < \epsilon) > \delta
$
for all $a\geq a_0$
and
for all $x \in (-L,L)$.  If we let $\epsilon=1$ and let $\delta$ be the corresponding value in $(0,1)$ then it follows by induction and the strong Markov property that
$
\mbb{P}_0^a(\tau > k) \leq (1-\delta)^k
$
 for every positive integer $k$ and $a \geq a_0$.
Therefore, if we choose $M'$ large enough so that $\sum_{k=M'}^\infty (k+1)(1-\delta)^k < \epsilon$, then   for $a \geq a_0$
\[
\mbb{E}_0^a\left[\tau\mbb{1}_{\{\tau> M'\}}\right] \leq \sum_{k=M'}^\infty (k+1)\mbb{P}_0^a(\tau > k) \leq \sum_{k=M'}^\infty(k+1)(1-\delta)^k < \epsilon.
\]
It is  also straightforward to see that there exists $M''>0$ such that $\mbb{E}_0^a\left[\tau\mbb{1}_{\{\tau> M''\}}\right] < \epsilon$ for $a < a_0$.  Letting $M=\max\{M',M''\}$ then gives the uniform integrability.
\end{proof}

\section{Examples}\label{examples}

This section is devoted to showing how the diffusivity $\mc{D}$ encodes surface microscopic structure when our operator $P$ represents a {\em random reflection}.  The structure of the section is as follows.  The first subsection gives a general outline for computing $\mc{D}$ independent of any given surface microscopic structure.  The second subsection computes $\mc{D}$ in the case that the surface of the walls is given by a periodic arrangement of focusing semicircles.  The last subsection gives $\mc{D}$ for certain parametric families of surfaces derived from the semicircle example of the previous subsection.

\subsection{General Technique}\label{gentech}

While in general, under Assumption \ref{assumption}, the diffusivity is given by
\[
t\mc{D} = \lim_{a\to\infty}\mbb{E}\left[\left(a^{-1}\sum_{j=0}^{N_{a,t}-1} Z^{I(a)}_j\right)^2\right],
\]
where $I(a)$ is the interval given in Theorem \ref{clt}, it is possible to  consider a significantly reduced truncation without altering the value of $\mc{D}$.

Let $\eta \in (0,1)$ and define $J(a) := \{x : \exp\left(\log^\eta a\right) < |x| < a/\log^\gamma a\}$.  The following proposition shows that we may use the truncated displacements $Z^{J(a)}$ in computing $\mc{D}$ so that in fact a vanishingly small cone of trajectories determine the diffusivity.
\begin{proposition}\label{prop51}
Under the  assumptions of Theorem \ref{clt},
\[
t\mc{D} = \lim_{a\to\infty}\mbb{E}\left[\left(a^{-1}\sum_{j=0}^{N_{a,t}-1} Z^{J(a)}_{j}\right)^2\right].
\]
\end{proposition}
\begin{proof}
As in the proof of Theorem  \ref{difflimitfiv}, the second moment on the right above may be expressed as the spectral integral
\[
\mbb{E}\left[\left(a^{-1}\sum_{j=0}^{N_{a,t}-1} Z^{J(a)}_{j}\right)^2\right]
= \int_{-1}^1 \frac{1+\lambda}{1-\lambda} \frac{\mbb{E}\left[N_{a,t}+O(1)\right]}{ah(a)}\left(\frac{h(a)}{a}\left\| Z^{J(a)}\right\|^2\right) \Pi_{Z^{J(a)}}(d\lambda).
\]
where $\Pi_{Z^{J(a)}}(d\lambda)=\left\|Z^{J(a)}\right\|^{-2}\left\langle Z^{J(a)}, \Pi(d\lambda)Z^{J(a)} \right\rangle$ is a probability measure on the spectrum of $P$.
Now observe that if $W := Z^{I(a)\setminus J(a)}$, then
\begin{align*}
\left\langle Z^{J(a)}, \Pi(d\lambda) Z^{J(a)} \right\rangle&= \left\langle Z^{I(a)}-W, \Pi(d\lambda)\left(Z^{I(a)}-W\right) \right\rangle \\
&=\left\langle Z^{I(a)}, \Pi(d\lambda) Z^{I(a)} \right\rangle- 2 \left\langle W, \Pi(d\lambda)Z^{I(a)} \right\rangle +\left\langle W, \Pi(d\lambda)W \right\rangle
\end{align*}
and $\|W\|_2^2=O(\ln^\eta a)$. Therefore,
$$\left|\left\|Z^{J(a)}\right\|_2^2-\left\|Z^{I(a)}\right\|_2^2\right|\leq 2\|W\|_2 \left\| Z^{I(a)}\right\|_2 +\|W\|_2^2=O(\ln^{\eta'} a)$$
for some $\eta'\in (0,1)$.  Multiplying both sides of the inequality by $h(a)/a=1/\ln a$ and taking the limit as $a\rightarrow \infty$
implies that $Z^{J(a)}$ and  $Z^{I(a)}$ grow at the same rate, and that in the limit formula for $\mathcal{D}$
we can use $Z^{J(a)}$ rather than  $Z^{I(a)}$.
\end{proof}

With the previous proposition in mind we use the shorthand notation $Z_a := Z^{J(a)}$ through the rest of the subsection without the risk of ambiguity.

If we expand
\[
\mbb{E}\left[\left(a^{-1}\sum_{j=0}^{N_{a,t}-1} Z_{a,j}\right)^2\right] = \mbb{E}\left[a^{-2}\sum_{j=0}^{N_{a,t}-1}Z_{a,j}^2\right] + \mbb{E}\left[2 a^{-2}\sum_{1 \leq i < j \leq N_{a,t}-1}Z_{a,i}Z_{a,j}\right],
\]
the limit of the first term has been shown in the proof of Theorem \ref{difflimitfiv} to be $\mathcal{D}_0$, which is independent
of the microstructure.
Let $C(a)$ be a function that  increases slower than $ah(a)$ but is otherwise to be determined by a specific microstructure.  For the second term we break the sum into two pieces as follows
\[
a^{-2}\mbb{E}\left[2 \sum_{1 \leq i < j \leq N_{a,t}-1}Z_{a,i}Z_{a,j}\right] = a^{-2}\mbb{E}\left[ \sum_{0 < |i-j| < C(a)} Z_{a,i}Z_{a,j}\right] + a^{-2}\mbb{E}\left[ \sum_{C(a) \leq |i-j| \leq N_{a,t}-1} Z_{a,i}Z_{a,j}\right].
\]
The first term on the right above will be determined by the microstructure, but the second term actually vanishes.

\begin{lemma}
We have $
\lim_{a\to\infty} a^{-2}\mbb{E}\left[ \sum_{C(a) \leq |i-j| \leq N_{a,t}-1} Z_{a,i}Z_{a,j}\right] = 0.
$
\end{lemma}
\begin{proof}
Observe that
\begin{align*}
a^{-2} \sum_{C(a) \leq |i-j| \leq N_{a,t}-1} \left|\mbb{E}\left[Z_{a,i}Z_{a,j}\right]\right|
&= a^{-2} \sum_{i=1}^{N_{a,t}-C(a)}\sum_{j=i}^{N_{a,t}-C(a)} \left|\mbb{E}\left[Z_{a,i}Z_{a,j+C(a)}\right]\right| \\
&= 2a^{-2}\, \mbb{E}\left[Z_{a,0}^2\right] \sum_{i=1}^{N_{a,t}-C(a)}\sum_{j=i}^{N_{a,t}-C(a)} \left|\corr(Z_{a,i},Z_{a,j+C(a)})\right| \\
&\leq 2a^{-2}\, \mbb{E}\left[Z_{a,0}^2\right] \sum_{i=1}^{N_{a,t}-C(a)}\sum_{j=i}^{N_{a,t}-C(a)} \left|\rho(C(a)+j-i)\right| \\
&\leq \frac{M}{ah(a)} \sum_{i=1}^{N_{a,t}-C(a)}\sum_{j=i}^{N_{a,t}-C(a)} \rho^{C(a)+j-i},
\end{align*}
where $M>0$ is a constant and $0 <\rho <1$ is the essential spectral radius.  Further,
\[
\sum_{i=1}^{N_{a,t}-C(a)}\sum_{j=i}^{N_{a,t}-C(a)} \rho^{C(a)+j-i}=  \frac{\rho^{C(a)}}{1-\rho}\left(N_{a,t}-C(a)-\frac{1-\rho^{N_{a,t}-C(a)+1}}{1-\rho}+1\right).
\]
Taking expectation and letting $a \to \infty$ then gives the result.
\end{proof}

\subsection{Semicircle Microstructure}

This section is devoted to computing the diffusivity of the channel system whose walls consist of a periodic focusing semicircle microscopic structure.  The first subsection is devoted to a thorough analysis of the operator $P$ for this geometry.  We show that $P$ is quasicompact and give an explicit formula for $P$ for a certain range of pre-collision angles.  The second subsection is devoted to the computation of $\mc{D}$ using the method outlined in the above section.

\subsubsection{A closer look at $P$}

To show that $P$ is quasicompact we employ the technique of conditioning (see \cite{fz} and \cite{fz2} for more details).  The general idea of conditioning in this setting is to obtain a compact operator by considering $P$ conditional on the event  that trajectories satisfy a given property.  As we will see, the compactness of this conditional operator will imply $P$ is quasicompact.

Let $Q$ denote the the microscopic cell bounded by the semicircle and its diameter such that $\partial Q = \Gamma_0 \cup \Gamma_1$ where we renormalize so that the diameter $\Gamma_0$ is identified with $[0,1]$ and $\Gamma_1$ is the semicircle with radius $1/2$.  Let $\Psi_\theta(r) \in [0,\pi]$ denote the angle between the outgoing vector $V$ and $\Gamma_0$ given that the trajectory enters $Q$ with angle $\theta$ at position $r$.

We define a measurable partition of $M= I \times V = [0,1]\times[0,\pi]$ as follows: let $M_1 \subset M$ be the subset of initial conditions whose billiard trajectories undergo exactly one or two collisions with $\partial Q$ before returning to $\Gamma_0$.  Define $P_1$ as $P$ conditional on the event $M_1$, and similarly define $P_2$ from $M_2 = M \setminus M_1$.  More precisely, if we let $M_j(\theta) = \{r \in I : (r,\theta)\in M_j\}$ and define $\alpha_j(\theta) = \lambda(M_j(\theta))$ for each $j=1,2$ and $\theta\in V$, then for each $f \in L^\infty(V,\mu)$, define
\[
(P_jf)(\theta) = \left\{\begin{array}{ll} \frac{1}{\alpha_j(\theta)}\int_{M_j(\theta)} f(\Psi_\theta(r))dr, & \alpha_j(\theta) \neq 0 \\ 0, & \alpha_j(\theta) = 0.\end{array}\right.
\]
We call $P_j$ the {\em conditional operators} associated to the partition $M_j$.  Note that it makes sense to write
$
Pf = \alpha_1P_1f + \alpha_2P_2f.
$
Let  $\mu_j$ be the measure  on $V$ defined such that $d\mu_j = \frac{\alpha_j}{(\lambda\times\mu)(M_j)}d\mu$.  It follows that $P_j$ is self adjoint on $L^2(V,\mu_j)$.

Next note that $P_1$ has an integral kernel.  Let $W_\theta^i = \{r \in I : (r,\theta) \in M_i\}$.  Then $W_\theta^1$ is the countable (or finite) union of open intervals $W_{\theta,j}$ for which the restriction $\Psi_{\theta,j} = \Psi_\theta|_{W_{\theta,j}}$ is a diffeomorphism from $W_{\theta,j}$ onto its image $V_{\theta,j}$.  Define $\Gamma_\theta(\varphi) = \sum_j \chi_{V_{\theta,j}}(\varphi)\Lambda_{\theta,j}(\varphi)^{-1}$, where
$
\Lambda_{\theta,j}(\varphi) = \frac{1}{2}\left|\Psi_\theta'(\Psi_{\theta,j}^{-1}(\varphi))\right|\sin\varphi.
$
Let
\[
\omega_1(\theta,\varphi) = \frac{(\lambda\times\mu)(M_1)\Gamma_\theta(\varphi)}{\alpha_1(\theta)\alpha_1(\varphi)}.
\]
It follows by way of change of variables that the operator $P_1$ on $L^2(V,\mu_1)$ is given by
\[
(P_1f)(\theta) = \int_V f(\varphi)\omega_1(\theta,\varphi)d\mu_1(\varphi).
\]
To show $P_1$ is compact, it will suffice to show $\omega_1$ is square integrable on $V \times V$.  To this end, we first look at the function $\Psi_\theta(r)$.  By the symmetry of the semicircle $\Psi_\theta(r)$ satisfies
\[
\Psi_\theta(r) = \pi - \Psi_{\pi-\theta}(1-r).
\]
It follows that it suffices to consider only $\theta \in (0,\pi/2)$.

\begin{proposition}\label{prop5} The function $\Psi_\theta(r)$ has the following properties:
\begin{enumerate}
\item Let $\theta \in (0,\pi/2)$ and let $n \geq 1$ be the number of collisions a trajectory with initial data $(r,\theta)$ makes with the semicircle.  Then
\[
\Psi_\theta(r) = \left\{\begin{array}{ll} 2n\sin^{-1}((2r-1)\sin\theta)+n\pi-\theta, & r\in [0,1/2] \\ 2n\sin^{-1}((2r-1)\sin\theta)-(n-2)\pi-\theta, & r \in(1/2,1]. \end{array}\right.
\]
\item Let $\theta \in (\pi/4, \pi/2)$.  Then $\Psi_\theta(r)$ has the following points of discontinuity in $[0,1]$:
\[
r_0^{(n)} = \frac{1}{2}-\frac{\sin[(n\pi-\theta)/(2n+1)]}{2\sin\theta}, \quad r_1^{(n)} = \frac{1}{2}+\frac{\sin[((n-1)\pi+\theta)/(2n+1)]}{2\sin\theta},
\]
$n \geq 1$.
\item Let $\theta \in (0, \pi/4)$.  Then $\Psi_\theta(r)$ has only one point of discontinuity given by
\[
r' = \frac{1}{2}+\frac{\sin\theta/3}{2\sin\theta}.
\]
\end{enumerate}
\end{proposition}
We remark that if $\theta \in (0,\pi/4)$ then at most two collisions in the semicircle are possible and hence the formula in 1. is only valid for $n=1,2$.  Further, we make more precise what is meant by the points of discontinuity given in 2. and 3.  If $r \in \left(r_0^{(1)},r_1^{(1)}\right)$, then the initial conditions $(r,\theta)$ give a billiard trajectory which makes only one intermediary collision.  And for $n \geq 2$, if $r \in \left(r_0^{(n)},r_0^{(n-1)}\right) \cup \left(r_1^{(n-1)}, r_1^{(n)}\right)$, then the initial conditions $(r,\theta)$ give a billiard trajectory which makes $n$ intermediary collisions.  For $0 < \theta < \pi/4$, the situation is simpler.  The initial conditions $(r,\theta)$ for $r \in (0, r')$, give one intermediary collision, and for $r$ in the complementary subinterval of $I$,  there are two intermediary collisions.  The proof of the proposition is by elementary trigonometry.

 \vspace{0.1in}
\begin{figure}[htbp]
\begin{center}
\includegraphics[width=2in]{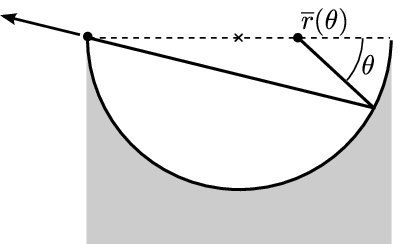}\ \ 
\caption{\small   Definition of $\overline{r}$. }
\label{illustration1}
\end{center}
\end{figure}

\begin{proposition}
The operator $P_1$ on $L^2(V,\mu_1)$ is compact.
\end{proposition}
\begin{proof}
As $P_1$ is given as an integral operator $(P_1f)(\theta) = \int_V f(\varphi)\omega_1(\theta,\varphi)d\mu_1(\varphi)$ it suffices to show that $\omega_1 \in L^2(V\times V, \mu_1 \times \mu_1)$.  By symmetry, we have the identity $\omega_1(\theta,\varphi) = \omega_1(\pi-\theta,\pi-\varphi)$ and the square of the $L^2$-norm of $\omega_1$ is
\[
2\int_0^{\pi/2}\int_V \omega_1^2(\theta,\varphi)d\mu_1(\varphi)d\mu_1(\theta) .
\]
Thus we show that the following two integrals are finite:
\begin{equation}\label{eq1}
\int_0^{\pi/4}\int_V \omega_1^2(\theta,\varphi)d\mu_1(\varphi)d\mu_1(\theta), \quad \int_{\pi/4}^{\pi/2}\int_V \omega_1^2(\theta,\varphi)d\mu_1(\varphi)d\mu_1(\theta).
\end{equation}
First consider the second integral.  Let $\pi/4 < \theta < \pi/2$.  Using the notation established in the previous proposition and further above, $W_\theta^1$ is given by the single interval $W_{\theta,1} = (r_0^{(1)},r_1^{(1)})$.  Correspondingly, $V_{\theta,1}=(\pi/3-\theta/3,\pi-\theta/3)$.  Moreover, $\Psi_{\theta,1}(r) = \pi-\theta+2\sin^{-1}((2r-1)\sin\theta)$.  It follows that
\[
|\Psi_\theta'(\Psi_{\theta,1}^{-1}(\varphi))| = \frac{4\sin\theta}{\sin\left(\frac{\varphi+\theta}{2}\right)}
\]

Now observe that
\begin{align*}
\int_{\pi/4}^{\pi/2}\int_V \omega_1^2(\theta,\varphi)d\mu_1(\varphi)d\mu_1(\theta) &= \int_{\pi/4}^{\pi/2} \int_V \frac{\Gamma_\theta(\varphi)^2}{\alpha_1(\theta)\alpha_1(\varphi)}d\mu(\varphi)d\mu(\theta) \\
&= \int_{\pi/4}^{\pi/2} \int_{V_{\theta,1}} \frac{\sin\theta}{\alpha_1(\theta)\alpha_1(\varphi)|\Psi_\theta'(\Psi_{\theta,1}^{-1}(\varphi))|^2\sin\varphi}d\varphi d\theta.
\end{align*}
Now because $\theta, \phi$ are bounded away from zero in the above integrals and because $\alpha_1(\theta) > 0$ for all $\theta$, it follows that the above integral is finite.

We show that the first integral in \eqref{eq1} is finite. Let $0 < \theta < \pi/4$.  Here $W_\theta^1$ is given by the single interval $W_{\theta,1} = (0,r')$ and $V_{\theta,1} =(\pi-3\theta,\pi-\theta/3)$.  Moreover, $\Psi_{\theta,1}(r) = \pi-\theta+2\sin^{-1}((2r-1)\sin\theta)$ as in the previous case.  Hence $|\Psi_\theta'(\Psi_{\theta,1}^{-1}(\varphi))|$ is also as above.  It follows that it suffices to show the following integral is finite:
\[
\int_0^{\pi/4}\int_{\pi-3\theta}^{\pi-\theta/3} \frac{\sin^2\left(\frac{\varphi+\theta}{2}\right)}{\sin\theta\sin\varphi}d\varphi d\theta.
\]
Note that for $\pi-3\theta < \varphi < \pi-\theta/3$, we have $\varphi \geq \theta$, which implies $\sin\left(\frac{\varphi+\theta}{2}\right) \leq \sin^2\varphi$.  Therefore,
\begin{align*}
\int_0^{\pi/4}\int_{\pi-3\theta}^{\pi-\theta/3} \frac{\sin^2\left(\frac{\varphi+\theta}{2}\right)}{\sin\theta\sin\varphi}d\varphi d\theta &\leq \int_0^{\pi/4}\int_{\pi-3\theta}^{\pi-\theta/3} \frac{\sin\varphi}{\sin\theta}d\varphi d\theta \\
&= \int_0^{\pi/4}\int_{\theta/3}{3\theta}\frac{\sin\varphi}{\sin\theta}d\varphi d\theta \\
&\leq \int_0^{\pi/4} \frac{1}{\sin\theta}\left(\frac{(3\theta)^2-(\theta/3)^2}{2}\right)d\theta,
\end{align*}
which is finite.
\end{proof}

Having shown that $P_1$ is compact we are now ready to show that $P$ is quasicompact.  This this a consequence of  the following general fact. (See Theorem 9.9 in  \cite{weid}.)
\begin{proposition}
Let $K$ and $T$ be bounded self adjoint operators on a Hilbert space and suppose that $K$ is compact.  Then the essential spectrum of $T+K$ is contained in the essential spectrum of $T$.  In particular, if $\norm{T+K} = 1$ and $\norm{T} < 1$, then the spectral gap $\gamma(T+K)$ of $T+K$ satisfies
\[
\gamma(T+K) \geq \min\{1-\norm{T}, \gamma(K)\}.
\]
\end{proposition}
The quasicompactness of $P$ then follows from letting $K = \alpha_1P_1$ and $T=\alpha_2P_2$ and noting that $1-\norm{T} \geq \inf \alpha_1 > 0$.

\subsubsection{Computation of $\mc{D}$}

In this subsection we use the shorthand $Z_a := Z^{J(a)}$ and $Z_{a,j} := Z^{J(a)}(\Theta_j)$ where $J(a)$ is as given at the start of Section \ref{gentech}.

We show, for the example of Figure \ref{shapes1}, that 
$
\mc{D} = \frac{4rv}{\pi}\frac{1+\zeta}{1-\zeta}
$
where
$
\zeta = -\frac{1}{4}\log 3.
$
Following the discussion in Section \ref{gentech} we aim to compute
\[
\lim_{a\to\infty} a^{-2}\sum_{0 < |i-j| < C(a)} \mbb{E}\left[ Z_{a,i}Z_{a,j}\right] 
\]
where we choose  $C(a) = \log_3\log a.
$
By stationarity, we are interested in computing $\mbb{E}\left[ Z_{a,0}Z_{a,j}\right]$ for $j < C(a)$.  That is, we are interested in the first $C(a)$ collisions of trajectories with a shallow initial angle.  Now although we've chosen these truncations of the between collision displacements because, as seen in the previous section, the transition probability kernel has a straightforward explicit formula for such pre-collision angles, they are not without their own complications.  It's clear that to keep track of the trajectories whose $j$th displacement falls out of the truncation range, making $Z_{a,0}Z_{a,j}$ vanish, quickly becomes intractable.  For this reason, we introduce the following so-called widened truncation, which, as we will show, will not change the diffusivity.

Consider a trajectory for which $Z_0 \in J(a)$.  I follows that, for large enough $a$,   
$$
|Z(3^j\Theta_0)| \leq |Z_j| \leq |Z(3^{-j}\Theta_0)|.
$$
for all $j < C(a)$. A straightforward estimate then shows that
$$
3^{-j}\exp(\log^\eta a) - C_1 < |Z_j| < 3^ja/\log^\gamma a + C_1
$$
where $C_1>0$.

Define $K(a) := \{x : 3^{-C(a)}\exp(\log^\eta a) < |x| < 3^{C(a)}a/\log^\gamma a + C_1\}$.  For the rest of the subsection we will consider the new truncated displacement $Z^{K(a)}$ as well as $Z^{J(a)}$, which we will continue to denote $Z_a$.
Note that each $Z^{K(a)}_j = Z^{K(a)}(\Theta_j)$ is nonzero.
Moreover, the following lemma shows that it suffices to compute $\mbb{E}\left[ Z_{a,i} Z^{K(a)}_j\right]$.

\begin{lemma} The following equality of limits holds:
\[
\lim_{a\to\infty} \sum_{0 < |i-j| < C(a)} a^{-2}\mbb{E}\left[ Z_{a,i}Z_{a,j}\right] = \lim_{a\to\infty} a^{-2}\sum_{0 < |i-j| < C(a)} \mbb{E}\left[ Z_{a,i}Z^{K(a)}_j\right] 
\]
\end{lemma}
\begin{proof}
Define $I_1 := \{x : 3^{-C(a)}\exp(\log^\eta a) < |x| < \exp(\log^\eta a)\}$ and $I_2 := \{x : a/\log^\gamma a < |x| < 3^{C(a)}a/\log^\gamma a + C_1\}$.
Note that $K(a) = I_1(a) \cup I_2(a)$.  Hence it suffices to show
\[
\lim_{a\to\infty} a^{-2}\sum_{0 < |i-j| < C(a)} \mbb{E}\left[Z_{a,i}\left(Z^{I_1(a)}_j+Z^{I_2(a)}_j\right)\right] = 0.
\]
Observe that
\begin{align*}
a^{-2}\left|\mbb{E}\left[Z_{a,i}\left(Z^{I_1(a)}_j+Z^{I_2(a)}_j\right)\right] \right|
&\leq a^{-2} \mbb{E}^{1/2}\left[Z_{a,i}^2\right]\left(\mbb{E}^{1/2}\left[\left(Z^{I_1(a)}_j\right)^2\right]+\mbb{E}^{1/2}\left[\left(Z^{I_2(a)}_j\right)^2\right]\right) \\
&= O\left(a^{-2}\log^{1/2}a ( \log\log a)^{1/2}\right),
\end{align*}
where the last step is due to $\mbb{E}\left[Z_{a,i}^2\right] = O(\log a)$ and   $\mbb{E}\left[\left(Z^{I_i(a)}_j\right)^2\right] = O(\log\log a)$.
Since the sum contains roughly $O(N_{a,t}C(a)) = O\left(\frac{a^2 \log\log a}{\log a}\right)$ such terms, the result follows.
\end{proof}

Let
$
q_i =  {Z_i}/{Z_{i-1}}
$
for $i \geq 1$.
The next lemma is a key technical tool in the computation.
\begin{lemma}\label{lemq}
Let $1 \leq j < O(|\log\theta|)$.  Then
$
\mbb{E}\left[q_1 \cdots q_j \mid \Theta_0 = \theta\right] = \zeta^j + AB^{j-1}\theta^2+O\left(\theta^4\right)$ for constants $A,B$ independent of $\theta$.
\end{lemma}
\begin{proof}
Observe that for $\theta$ sufficiently small, using the integral kernel for $P$ derived in the previous subsection
\begin{align*}
\mbb{E}\left[q_1 \mid \Theta_0 = \theta\right]
&=\frac{1}{\cot\theta}\left(\int_{\pi-3\theta}^{\pi-\theta/3} \cot\varphi \frac{\cos\left(\frac{\varphi+\theta-\pi}{2}\right)}{4\sin\theta}\,d\varphi
+ \int_{\theta/3}^{3\theta} \cot\varphi \frac{\cos\left(\frac{\varphi+\theta}{4}\right)}{8\sin\theta}\,d\varphi\right)\\
&= \frac{1}{\cos\theta}\left(-\frac{1}{4}\int_{\theta/3}^{3\theta}\cot\varphi\cos\left(\frac{\varphi-\theta}{2}\right)\,d\varphi +\frac{1}{8}\int_{\theta/3}^{3\theta}\cot\varphi\cos\left(\frac{\varphi+\theta}{4}\right)\,d\varphi\right).
\end{align*}
One may check that
\[
\frac{1}{\cos\theta} \int_{\theta/3}^{3\theta}\cot\varphi\cos\left(\frac{\varphi-\theta}{2}\right)\,d\varphi = 2\log 3 + D_1\theta^2 + O(\theta^4),
\]
where $D_1$ is a constant, and likewise for the second integral above, albeit with a constant different from $D_1$.  The case $j=1$ then follows.

Let $V_\theta = (\theta/3, 3\theta) \cup (\pi-3\theta,\pi-\theta/3)$.  Observe that
\begin{align*}
\mbb{E}\left[q_1q_2\mid \Theta_0 = \theta\right]
&= \mbb{E}\left[\mbb{E}\left[q_1 q_2 \mid \Theta_0 = \theta, \Theta_1 = \varphi\right] \mid \Theta_0 = \theta\right] 
= \int_{V_\theta} \mbb{E}\left[q_1 q_2 \mid \Theta_0 = \theta, \Theta_1 = \varphi\right] \,P(\theta,d\varphi) \\
&= \int_{V_\theta} \frac{z(\varphi)}{z(\theta)} \mbb{E}\left[q_2  \mid \Theta_0 = \theta, \Theta_1 = \varphi\right] \,P(\theta,d\varphi) 
= \int_{V_\theta} \frac{z(\varphi)}{z(\theta)}\left(\zeta + A\varphi^2 + O(\varphi^4)\right) P(\theta,d\varphi) \\
&= \zeta\,\mbb{E}\left[q_1 \mid \Theta_0 = \theta\right] + \int_{V_\theta}\frac{z(\varphi)}{z(\theta)}\left(A\varphi^2 + O(\varphi^4)\right) P(\theta,d\varphi) \\
&= \zeta^2+A\theta^2+O\left(\theta^4\right) + \int_{V_\theta}\frac{z(\varphi)}{z(\theta)}\left(A\varphi^2 + O(\varphi^4)\right) P(\theta,d\varphi).
\end{align*}
Next, one may check that
\begin{align*}
\int_{V_\theta}\frac{z(\varphi)}{z(\theta)}\left(A\varphi^2 + O(\varphi^4)\right) P(\theta,d\varphi) &= 
\frac{1}{\cos\theta}\left(-\frac{1}{4}\int_{\theta/3}^{3\theta}\cot\varphi\left(A\varphi^2 + O(\varphi^4)\right)\cos\left(\frac{\varphi-\theta}{2}\right)\,d\varphi\right. \\&     \left.+\frac{1}{8}\int_{\theta/3}^{3\theta}\cot\varphi\left(A\varphi^2 + O(\varphi^4)\right)\cos\left(\frac{\varphi+\theta}{4}\right)\,d\varphi\right) = AB\theta^2+O\left(\theta^4\right)
\end{align*}
The case $j=2$   follows.  The rest of the argument follows by a similar induction argument.
\end{proof}

With the above lemma in place we are ready to compute the correlations.

\begin{lemma}
For $1 \leq j < C(a)$,
$
\mbb{E}\left[Z_{a,0}Z^{K(a)}_j\right] = 4r^2\zeta^j \Lambda(a) + AB^{j-1}\Gamma(a),
$
where $\Lambda(a) \sim \log a$, $\Gamma(a) = O\left(\exp(-2\log^\eta a)\right)$, and $A, B$ are the constants given in Lemma \ref{lemq}.
\end{lemma}
\begin{proof}
Recall the interval $J(a) = \{x : \exp\left(\log^\eta a\right) < |x| < a/\log^\gamma a\}$.  Observe that
\begin{align*}
\mbb{E}\left[Z_{a,0}Z^{K(a)}_j\right] &= \int_0^\pi Z_a(\theta)^2\mbb{E}\left(q_1 \cdots q_j \mid \Theta_0 = \theta\right) \,\mu(d\theta) \\
&= 2\int_0^\pi Z(\theta)^2\mbb{E}\left(q_1 \cdots q_j \mid \Theta_0 = \theta\right)  \mbb{1}_{J(a)}\,\mu(d\theta) \\
&= 2\int_0^\pi Z(\theta)^2\left(\zeta^j + AB^{j-1}\theta^2+ O\left(\theta^4\right)\right)  \mbb{1}_{J(a)}\,\mu(d\theta) .
\end{align*}
We remark that because we are only considering here $\theta$ such that $\exp(\log^\eta a) < Z(\theta) < a/\log^\gamma a$, and hence $|\log \theta| < C\log a$ for some constant $C$, it follows that $j < C(a)$ is sufficiently small so that we may apply Lemma \ref{lemq}.

It is straightforward to compute
$
\int_0^\pi Z(\theta)^2 \mbb{1}_{J(a)}\,\mu(d\theta)  \sim 2r^2\log a.
$
And moreover
\[
\int_0^\pi Z(\theta)^2 O(\theta^2)\mbb{1}_{J(a)}\,\mu(d\theta)  = \int_0^\pi O(\theta) \mbb{1}_{J(a)} \,d\theta = O\left((\exp(\log^\eta a))^{-2}\right).
\]
The result now follows.
\end{proof}

The summation of correlations is the final piece to our computation.

\begin{proposition}
Let $U(a) = \sum_{0 < |i-j| < C(a)} \mbb{E}\left[ Z_{a,i}Z^{K(a)}_j\right]$.  Then
$
\lim_{a\to\infty} a^{-2}U(a) = \frac{8trv}{\pi} \frac{\zeta}{1-\zeta}.
$
\end{proposition}
\begin{proof}
Observe that
\begin{align*}
U(a)
&= 2\left[(N_{a,t}-C(a)+1)\sum_{i=1}^{C(a)-1}\mbb{E}\left[Z_{a,0}Z^{K(a)}_i\right] + \sum_{i=1}^{C(a)-2}\sum_{j=1}^{C(a)-i-1}\mbb{E}\left[Z_{a,0}Z^{K(a)}_j\right]\right]\\
&= 2\left[(N_{a,t}-C(a)+1)\sum_{i=1}^{C(a)-1} \left(4r^2\zeta^j \Lambda(a)+AB^{j-1}\Gamma(a)\right)\right. \\&\,\,\,\,\,\,\,\,+  \left.\sum_{i=1}^{C(a)-2}\sum_{j=1}^{C(a)-i-1}\left(4r^2\zeta^j \Lambda(a) + AB^{j-1}\Gamma(a)\right)\right]\\
&= 2\left[A\Gamma(a)\left((N_{a,t}-C(a)+1)\frac{B^{C(a)}-B}{B(B-1)} + \frac{B^{C(a)}+B(C(a)-2)-B^2(C(a)-1)}{B(B-1)^2}\right) \right.\\&\,\,\,\,\,\,\,\,+ \left. 4r^2\Lambda(a)\left((N_{a,t}-C(a)+1)\frac{\zeta^{C(a)}-\zeta}{\zeta-1} + \frac{\zeta^{C(a)}+\zeta(C(a)-2)-\zeta^2(C(a)-1)}{(\zeta-1)^2}\right)\right].
\end{align*}
Dividing by $a^2$ and letting $a \to \infty$ gives the result.
\end{proof}

The value of $\mc{D}$ then follows from adding the value from the proposition above to $4trv/\pi$ as discussed in Section \ref{gentech}.

\subsection{Parametric Families}

In this section we consider three different parametric families which are derived from the semicircle.  Of primary interest will be how the diffusivity of the limiting process for each family changes as a function of the parameter.


We begin with the family formed by adding a middle wall of height $h$ to the semicircle as shown in the figure above.  Suppose $h < 1/2$; that is, the wall does not extend to the center of the semicircle.  It is apparent by inspection that trajectories with a sufficiently small pre-collision angle will never, so to speak, notice the middle wall.  And moreover, those trajectories with initial data $(r,\theta)$ that do notice the wall will behave like trajectories in the semicircle with no wall with initial data $(1-r,\pi-\theta)$ by symmetry.  That is, if $\Psi_\theta^h(r)$ denotes the post-collision angle of a trajectory with initial data $(r,\theta)$ in the middle wall geometry and similarly $\Psi_\theta(r) = \Psi_\theta^0(r)$, then
$
\Psi_\theta^h(r) = \Psi_{\pi-\theta}(1-r)
$
for initial conditions $(r,\theta)$ for which the trajectory hits the middle wall.  It follows that the operator $P_h$ for the middle wall geometry is quasicompact.  Moreover, as pointed out earlier, the diffusivity $\mc{D}_h$ for $P_h$ depends only on trajectories with arbitrarily shallow angles by the formula in Proposition \ref{prop51}.  It follows that the diffusivity $\mc{D}(h)$ is constant and equal to the diffusivity for the semicircle with no wall for all $h < 1/2$.

Further, using the symmetry $\Psi_\theta^h(r) = \Psi_{\pi-\theta}(1-r)$ and following the proof given for the semicircle, the diffusivity $\mc{D}({1/2})$ for the geometry with a middle that extends exactly to the center of the semicircle is given by
$
\mc{D}({1/2}) = \frac{4rv}{\pi} \frac{1-\zeta}{1+\zeta}
$
where $\zeta = -\frac{1}{4}\log 3$ as in the semicircle.  We summarize these facts as follows.

\begin{proposition}
For the middle wall modification of the semicircle with middle wall height $h$,  
$
\mc{D}(h) = \mc{D},
$ for $h<1/2$, 
where $\mc{D}$ is the diffusivity for the semicircle with no wall, and  for $h=1/2$ the diffusivity is given by
\[
\mc{D}({1/2}) = \frac{4rv}{\pi} \frac{1-\zeta}{1+\zeta},
\]
where $\zeta = -1/4\log 3$.
\end{proposition}

It is interesting to note that the diffusivity is not a continuous function of the parameter.  We also remark that the case $h > 1/2$ when the middle wall extends outside the semicircle requires a different analysis altogether, which we leave for a future paper.


Next we look to the geometry formed by splitting the semicircle and adding a flat bottom of length $h \in (0,1)$ as shown above.  We renormalize the size of the semicircles so that they have radius $(1-h)/2$.  We also establish the notation
$
a=(1+h)/{2}$, $b = (1-h)/{2}.
$
While qualitatively similar to the semicircle, the angle function $\Psi_\theta^h(r)$ requires a new detailed analysis, which we sketch here.

We begin by noting that by symmetry it again suffices to consider only pre-collision angles $\theta \in (0,\pi/2)$.  And moreover, for sufficiently small pre-collision angles $\theta$ at most two intermediary collisions are possible within the cell.  In the discussion that follows we consider only such $\theta$.  Let $r' \in (0,1)$ be the point of entry for which $\Psi_\theta^h(r)$ is discontinuous.  That is, $r'$ is chosen such that trajectories with initial data $(r,\theta)$ for $r \in (0,r')$ experience one intermediary collision, and for those with $r \in (r',1)$, there are two intermediary collisions.  It follows that
$
\Psi_\theta^h(r) = \pi - \theta -2\beta_1(r)$ for $ r \in (0,r')
$
where $\beta_1 = \beta_1(r)$ satisfies
$
b\sin\beta_1 = (a-r)\sin\theta.
$
We may also characterize $r'$ as the value of $r$ that satisfies
$
a\sin(\theta-2\beta_1) = b\sin\beta_1.
$
From these observations it follows that
\[
\Psi_\theta^h(0) = \pi - \frac{3+h}{1-h}\theta + O\left(\theta^3\right), \quad \lim_{r \to (r')^-} \Psi_\theta^h(r) = \pi-\frac{1-h}{3+h}\theta+  O\left(\theta^3\right).
\]
Following the notation established in the discussion on the semicircle,
\[
V_{\theta,1} = \left[ \pi - \frac{3+h}{1-h}\theta + O\left(\theta^3\right), \pi-\frac{1-h}{3+h}\theta+  O\left(\theta^3\right) \right].
\]
Moreover, if we let $\Theta = \Psi_\theta(r)$ it follows from implicit differentiation that
\[
\Theta'(r) = \frac{1}{1-h}\frac{4\sin\theta}{\cos\left(\frac{\Theta+\theta-\pi}{2}\right)}
\]
which is the corresponding value in the semicircle case except for the factor of $(1-h)^{-1}$.

 \vspace{0.1in}
\begin{figure}[htbp]
\begin{center}
\includegraphics[width=3in]{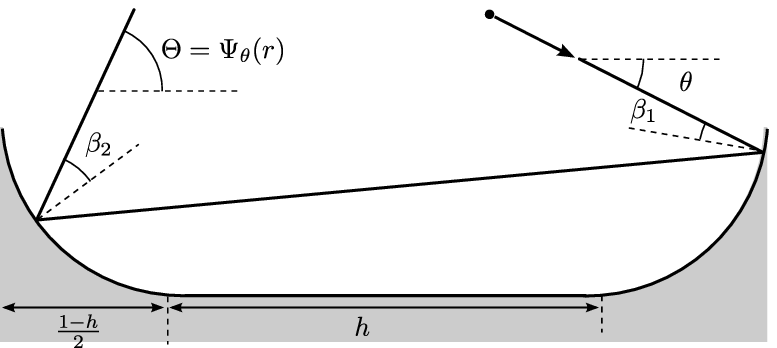}\ \ 
\caption{\small     }
\label{illustration2}
\end{center}
\end{figure}

For $r \in (r',1)$, we have
$
\Psi_\theta^h(r) = 2\beta_1(r) + 2\beta_2(r) - \theta,
$
where $\beta_1(r)$ is as above and $\beta_2 = \beta_2(r)$ satisfies
$
b\sin\beta_2 = h\sin(2\beta_1-\theta) + \sin\beta_1.
$
It follows by symmetry that
\[
V_{\theta,2} = \left[\frac{1-h}{3+h}\theta+  O\left(\theta^3\right),  \frac{3+h}{1-h}\theta + O\left(\theta^3\right) \right].
\]
Further, by implicit differentiation again
\[
\Theta'(r) = \frac{1+h}{(1-h)^2}\frac{8\sin\theta}{\cos\left(\frac{\Theta+\theta}{4}\right)},
\]
which again resembles the semicircle case but for the factor of $(1+h)(1-h)^{-2}$.
As in the computation of the diffusivity for the semicircle we find the following.

\begin{proposition}
For the flat bottom of length $h \in (0,1)$ modification of the semicircle, we have the following value of diffusivity as a function of $h$:
\[
\mc{D}(h) = \frac{4rv}{\pi} \frac{1+\zeta_h}{1-\zeta_h},
\]
where
\[
\zeta_h = -\frac{1+3h}{4}\frac{1-h}{1+h}\log\frac{3+h}{1-h}.
\]
\end{proposition}
Notice that the limiting case $h=0$ gives the diffusivity of the semicircle while $h=1$ gives $\zeta_h=0$ which gives the diffusivity in the case that at each post-collision angle is chosen independently according to the distribution $\mu$.

The final family of interest is formed by adding a flat side of length $h \in (0,1)$ between semicircles as shown above.  The key observation here is that the operator $P_h$ corresponding to such a geometry can be thought of as a sum of conditional operators given by
$
P_h = (1-h)P + hI
$
where $P$ is the operator corresponding to the semicircle geometry and $I$ is the identity operator.  Such an operator is a generalized Maxwell-Smoluchowski model as discussed in the introduction.  Next notice that since $P$ and $I$ commute the spectra of $P_h$ and $P$ are the same, and in fact every spectral value of $P_h$ is given by
$
(1-h)\lambda + h
$
for some unique spectral value $\lambda$ of $P$.  It follows that
\[
\int_{-1}^1 \frac{1+\lambda}{1-\lambda} \Pi_{a}^h(d\lambda) = \int_{-1}^1 \frac{1+(1-h)\lambda + h}{1-(1-h)\lambda - h} \Pi_{a}(d\lambda) = \int_{-1}^1 \left(\frac{1}{1-h} \frac{1+\lambda}{1-\lambda} + \frac{h}{1-h}\right) \Pi_{a}(d\lambda),
\]
where $\Pi_a$ is the spectral measure as given in the statement of Theorem \ref{clt}, and similary $\Pi_a^h$ is derived from the projection valued measure $\Pi^h$ associated to $P_h$ given by the spectral theorem.
The following proposition then follows from the discussion above and the spectral formulation of the diffusivity.
\begin{proposition}
For the flat top of length $h \in (0,1)$ modification of the semicircle, we have the following value of diffusivity as a function of $h$:
\[
\mc{D}(h) = \frac{1}{1-h} \mc{D} + \frac{h}{1-h} \frac{4rv}{\pi} = \frac{4rv}{\pi(1-h)}\left(h + \frac{1+\zeta}{1-\zeta}\right),
\]
where $\mc{D}$ is the diffusivity for the unmodified semicircle and $\zeta = -1/4 \log 3$.
\end{proposition}

\section{Additional proofs}
We collect here proofs of some of the more technical  propositions and lemmas from earlier parts of the paper.
\subsection{Sketch of proof of Proposition \ref{volumes}}\label{proofvolumes}

The proposition results from  tedious but elementary and straightforward integrations. We show a few steps  to convey the flavor.

By identifying the tangent space to   $\mathcal{C}$ at  $q\in \partial \mathcal{C}$   with $\mathbb{R}^k\oplus \mathbb{R}^{n-k}$,
the unit normal vector $n$ is identified 
with  $e_n=(0, \dots, 0, 1)$  and   $\mathbb{H}_q$ is identified with 
 the half-space consisting of vectors $v=(v_1, v_2)\in \mathbb{R}^k\oplus \mathbb{R}^{n-k}$ such that $\langle v, e_n \rangle >0$. For such a post-collision velocity $v$, the point of next collision with the channel boundary is $q+\tau v$, where
 $\tau_b=2 r \langle v, e_n \rangle /|v_2|^2$ is   the   time interval  between the two collisions and $Z(v)=2r\langle v, e_n \rangle  v_1/|v_2|^2$ is the between-collisions
 displacement vector in ``horizontal'' factor $\mathbb{R}^k$ as represented in Figure \ref{cylinder}.  
We assume that $\nu$ has the following general form: $$d\nu(v)=C\langle v, n\rangle f\left(|v|^2\right)\, dV(v),$$
in which $f(x)$ is a nonnegative function on $[0,\infty)$ such that $\int_0^\infty x^n f\left(x^2\right)\, dx<\infty.$
We allow $f$ to be distributional (a delta-measure concentrated on a fixed value of $|v|$) so as to include the
case where $\nu$ is supported  on an  hemisphere. Let  $S_+:=S^{n-1}_+$ be  the unit
hemisphere in $\mathbb{H}$. Then $\mathbb{H}$ is diffeomorphic to $S_+\times (0,\infty)$ under polar coordinates $v\mapsto (w,\rho)$, where $w=v/|v|$ and $\rho=|v|$, and $dV(v)=\rho^{n-1}\, dV_{\text{\tiny sph}}(w)\, d\rho$. 
Also define the notations  $S_{+a}:=S_+ \cap \mathbb{H}(a)$ and 
  $E(a):= E_\nu\left[|Z|^2, \mathbb{H}(a)\right]$. Then 
$$ E(a)=C\int_0^\infty \int_{S_{+a}} \rho^{n} f\left(\rho^2\right)\  |Z(w)|^2 \langle w, n\rangle \, dV_{\text{\tiny sph}}(w) \, d\rho=C' \int_{S_{+a}} |Z(w)|^2 \langle w,n\rangle \, dV_{\text{\tiny sph}}(w),$$ 
where we have used that $Z(v)=Z(w)$. Here and below,  $C, C', C''$ are  positive constants that can be obtained explicitly. 

The image of    $S_{+a}$ under the projection map 
 $\pi: w\in S_+ \mapsto \overline{w} \in \mathbb{B}^{n-1}$
  contains the ball $\mathbb{B}^{n-1}_a$ of radius $\left(1+(2/a)^2\right)^{-1/2}$, and it is equal to $\mathbb{B}^{n-1}_a$ when
  $n-k=1$.
The volume elements $dV_{\text{\tiny sph}}$  on the hemisphere and $dV_{n-1}$ on the ball are related under $\pi$ by
  $\langle w, n\rangle dV_{\text{\tiny sph}}(w)=dV_{n-1}(\overline{w}).$ Also let $x_1=\pi_1(x)$  be 
  the natural projection from  the  unit ball in $\mathbb{R}^{n-1}$ to the unit ball $\mathbb{B}^{k}$  in the horizontal factor $\mathbb{R}^{k}$,
  and $x_2=\pi_2(x)$ the projection to the complementary factor $\mathbb{R}^{n-k-1}$ in $\mathbb{R}^{n-1}$.
Then  the limit as $a\rightarrow \infty$  of $\int_{S_{+a}} |Z(w)|^2 \langle w,n\rangle \, dV_{\text{\tiny sph}}(w)$ is, up to   multiplicative positive constant, the same as the limit of 
\begin{align*}\int_{\mathbb{B}^{n-1}_a}\frac{\left(1-|x_1|^2-|x_2|^2\right)|x_1|^2}{\left(1-|x_1|^2\right)^2}\, dV_{n-1}(x)
& =  \int_{\mathbb{B}^{k}_a}  \int_{\mathbb{B}^{n-k-1}_{|x_1|}}  \frac{|x_1|^2\left(1-|x_1|^2-|x_2|^2\right)}{\left(1-|x_1|^2\right)^2} \, dV_{n-k-1}(x_2)\, dV_{k}(z_1)\\
&=\frac{2 \text{Vol}(\mathbb{B}^{n-k-1})}{n-k+1}\int_{\mathbb{B}^k_\epsilon} |x_1|^2 \left(1-|x_1|^2\right)^{\frac{n-k-3}{2}}\, dV_k(x_1)\\
&=\frac{k}{n-k+1}\text{Vol}\left(\mathbb{B}^{n-k-1}\right)\text{Vol}\left(\mathbb{B}^{k}\right) \int_0^{\frac{a^2}{4+a^2}} s^{\frac{k}{2}}(1-s)^{\frac{n-k-3}{2}}\, ds
\end{align*}
where the iterated   integrals  were carried out in polar coordinates, in which volume   elements are related by  
  $dV_{n}(v)=|v|^{n-1} \, dV_{\text{\tiny sph}}(v/|v|)\, d|v|$. The limit of the  remaining integral, as $a$ goes to $\infty$,  is a Beta-function of the
   exponents of $s$ in the integrand; it,  and the volumes of unit balls, can be written in terms of Gamma-functions and further simplified.

The expected values of $\tau_b(v)=2r\langle v, n\rangle/|v_2|^2$ are  shown by similar computations to take the form 
$$E_\mu\left[\tau_b\right]= \frac{2 r}{s} \frac{\Gamma\left(\frac{n+1}{2}\right)}{\pi^{\frac{n-1}{2}}} I(n,k), \ \  E_{\mu_\beta}\left[\tau_b\right]=
r \left(2\beta M\right)^{\frac12} \frac{\Gamma\left(\frac{n}{2}\right)}{\pi^{\frac{n-1}{2}}} I(n,k),$$
where $s$ is the speed (or radius) of the hemisphere  on which $\mu$ is supported and
$$I(n,k):=\int_{S_+}  \frac{\langle x, n\rangle^2}{|x_2|^2}\, dV_{\text{\tiny sph}}(x)=\frac{\pi^{\frac{n}2}}{(n-k)\Gamma\left(\frac{n}{2}\right)},$$
$S_+$ being the hemisphere of radius $1$ in  $\mathbb{H}$.

Finally, observed that  
$\mathbb{E}_\nu\left[\left(Z^u_a\right)^2\right]=\mathbb{E}_\nu\left[|Z_a|^2\right]/k$
since  $\nu$ is rotationally symmetric in the $\mathbb{R}^k$ subspace. The proposition is now a consequence of these observations.

\subsection{Proofs of CLT lemmas}\label{lemproofs}

\begin{proof}[Proof of Lemma \ref{lem2}]
We again employ  the notation  $D_k^l:=\sum_{j=k}^{l-1}Z_{a,j}-\sum_{j=k}^{n_{a,t}-1}Z_{a,t}.$
Since
$
\sum_{j=0}^{N_{a,t}-1}Z_{a,j} = \sum_{j=0}^{n_{a,t}-1}Z_{a,j} + D_0^{N_{a,t}}
$
it suffices to show
$
D_0^{N_{a,t}} \to 0
$
in probability.  Before getting to this directly, we begin with a few technical remarks to be used in what follows.

Let $\nu > 1$ and define
\[
n_1 = \left[\left(1-\frac{\epsilon^{2\nu}}{2}\right)n_{a,t}\right]+1, \quad n_2 = \left[\left(1+\frac{\epsilon^{2\nu}}{2}\right)n_{a,t}\right]-1.
\]
Then
$
\mbb{P}(n_1 \leq N_{a,t} \leq n_2) \geq 1-\epsilon
$
for $a$ large enough.  Let
$
a^* = \epsilon^\nu a, \quad n_{a,t}^* = n_{a^*,t}.
$
Observe that
\[
\epsilon^{2\nu}n_{a,t} \leq \epsilon^{2\nu} \frac{ah(a)t}{\mathbb{E}[\tau_b]} \leq \frac{a^*h(a^*)t}{\mathbb{E}[\tau_b]} \leq n_{a,t}^*+1.
\]
Therefore,
$
n_2 - n_1 \leq \epsilon^{2\nu}n_{a,t} - 1 \leq n_{a,t}^*.
$
We are now ready to address the convergence in probability.  Observe that
\begin{align*}
\mbb{P}\left(\left|D_1^{N_{a,t}}\right| > a\epsilon\right) 
&= \sum_{j=1}^\infty \mbb{P}\left(N_{a,t}=j, \left|D_1^{N_{a,t}}\right| > a\epsilon\right) \\
&= \sum_{j=n_1}^{n_2} \mbb{P}\left(N_{a,t}=j, \left|D_1^{N_{a,t}}\right| > a\epsilon\right) + \sum_{j\not\in [n_1,n_2]} \mbb{P}\left(N_{a,t}=j, \left|D_1^{N_{a,t}}\right| > a\epsilon\right)
\end{align*}
Notice for the second term above that
$$
\sum_{j\not\in [n_1,n_2]} \mbb{P}\left(N_{a,t}=j, \left|D_1^{N_{a,t}}\right| > a\epsilon\right)  \leq \sum_{j\not\in [n_1,n_2]} \mbb{P}\left(N_{a,t}=j\right) = 1-\mbb{P}(n_1 \leq N_{a,t} \leq n_2) \leq \epsilon.
$$
For the first term observe that
$$
\sum_{j=n_1}^{n_2} \mbb{P}\left(N_{a,t}=j, \left|D_1^{N_{a,t}}\right| > a\epsilon\right)  \leq \sum_{j=n_1}^{n_2} \mbb{P}\left(N_{a,t}=j, \max_{n_1 \leq k \leq n_2} \left|D_1^{k}\right| > a\epsilon\right) 
\leq \mbb{P}\left(\max_{n_1 \leq k \leq n_2} \left|D_1^{k}\right| > a\epsilon\right).
$$
And notice that
\begin{align*}
\mbb{P}\left(\max_{n_{a,t} < k \leq n_2} \left|D_1^{k}\right| > a\epsilon\right) &= \mbb{P}\left(\max_{n_{a,t} < k \leq n_2} \left|\sum_{i=n_{a,t}+1}^k Z_{a,i}\right| > a\epsilon\right) \\
&= \mbb{P}\left(\max_{1 \leq k \leq n_2-n_{a,t}} \left|\sum_{i=1}^k Z_{a,i}\right| > a\epsilon\right) 
\leq \mbb{P}\left(\max_{1 \leq k \leq n_2-n_1} \left|\sum_{i=1}^k Z_{a,i}\right| > a\epsilon\right).
\end{align*}
Similarly,
\[
\mbb{P}\left(\max_{n_1 \leq k < n_{a,t}} \left|D_1^{k}\right| > a\epsilon\right) \leq \mbb{P}\left(\max_{1 \leq k \leq n_2-n_1} \left|\sum_{i=1}^k Z_{a,i}\right| > a\epsilon\right).
\]
It then follows that
$$
\mbb{P}\left(\max_{n_1 \leq k \leq n_2} \left|D_1^{k}\right| > a\epsilon\right) \leq 2 \mbb{P}\left(\max_{1 \leq k \leq n_2-n_1} \left|\sum_{i=1}^k Z_{a,i}\right| > a\epsilon\right) 
\leq 2 \mbb{P}\left(\max_{1 \leq k \leq n_{a,t}^*} \left|\sum_{i=1}^k Z_{a,i}\right| > a\epsilon\right).
$$
Next observe that by the symmetry of $Z_{a,i}$,
\[
\frac{\mbb{P}\left(\left|\sum_{i=1}^{n_{a,t}^*} Z_{a,i}\right| > a\epsilon, \max_{1 \leq k \leq n_{a,t}^*} \left|\sum_{i=1}^k Z_{a,i}\right| > a\epsilon\right)}{\mbb{P}\left(\max_{1 \leq k \leq n_{a,t}^*} \left|\sum_{i=1}^k Z_{a,i}\right| > a\epsilon\right)} = \mbb{P}\left(\left|\sum_{i=1}^{n_{a,t}^*} Z_{a,i}\right| > a\epsilon \;\middle\vert\; \max_{1 \leq k \leq n_{a,t}^*} \left|\sum_{i=1}^k Z_{a,i}\right| > a\epsilon\right) \geq \frac{1}{2},
\]
which implies
\[
\mbb{P}\left(\max_{1 \leq k \leq n_{a,t}^*} \left|\sum_{i=1}^k Z_{a,i}\right| > a\epsilon\right) \leq 2 \mbb{P}\left(\left|\sum_{i=1}^{n_{a,t}^*} Z_{a,i}\right| > a\epsilon, \max_{1 \leq k \leq n_{a,t}^*} \left|\sum_{i=1}^k Z_{a,i}\right| > a\epsilon\right) \leq 2 \mbb{P}\left(\left|\sum_{i=1}^{n_{a,t}^*} Z_{a,i}\right| > a\epsilon\right).
\]
We summarize and observe that by assumption
\begin{align*}
\sum_{j=n_1}^{n_2} \mbb{P}\left(N_{a,t}=j, \left|D_1^{N_{a,t}}\right| > a\epsilon\right) &\leq 4 \mbb{P}\left(\left|\sum_{i=1}^{n_{a,t}^*} Z_{a,i}\right| > a\epsilon\right) = 4\mbb{P}\left(\left|\sum_{i=1}^{n_{a,t}^*} Z_{a,i}\right| > a^*\epsilon^{1-\nu}\right) \\
&\to \frac{8}{\sqrt{2\pi t\mc{D}}} \int_{\epsilon^{1-\nu}}^\infty e^{-x^2/2t\mc{D}}\,dx,
\end{align*}
as $a\to\infty$.  In letting $\epsilon \to 0$ the last line above vanishes.
\end{proof}

\begin{proof}[Proof of Lemma \ref{lem3}]
Let $I_1(a) := \{ x : a/\log^\gamma(a) < |x| < a\log\log a\}$ and $I_2 := \{x : |x| > a \log\log a\}$.  
It will suffice to show that each of the sums $\sum_{j=0}^{N_{a,t}-1} a^{-1}Z^{I_i(a)}_j$ for $i=1,2$ converges to zero in probability.

We first consider the truncation by $I_2$.  As there is no ambiguity, we use the shorthand $Z_a := Z^{I_2}$ and $Z_{a,j} = Z^{I_2}(\Theta_j)$.  Let $\epsilon > 0$ and observe that
\begin{align*}
\mbb{P}\left(\left|\sum_{j=0}^{N_{a,t}-1} Z_{a,j}\right| > a\epsilon \right) &\leq \mbb{P}\left(\sum_{j=0}^{N_{a,t}-1} \left|Z_{a,j}\right| > a\epsilon\right)= \sum_{N=1}^\infty \mbb{P}\left(\sum_{j=0}^{N} \left|Z_{a,j}\right| > a\epsilon\right) \mbb{P}\left(N_{a,t}=N\right) \\
&\leq \sum_{N=1}^\infty \mbb{P}\left( Z_j \in I_2\text{ for some $j$}, 
0 \leq j \leq N-1
\right) \mbb{P}(N_{a,t}=N) \\
&\leq \sum_{N=1}^\infty \sum_{j=0}^{N-1}\mbb{P}(Z_j \in I_2)\mbb{P}(N_{a,t}=N) \\&= \mbb{E}\left[N_{a,t}\right] \mbb{P}(Z_0 \in I_2).
\end{align*}
Next note that
\begin{align*}
\mbb{P}(Z_0 \in I_2)
&= \int_{0}^{\cot^{-1}(a\log\log a/(2r))}\sin\theta\,d\theta \\
&= 1-\cos\left(\cot^{-1}\left(a\log\log a/(2r)\right)\right) = O\left((a\log\log a)^{-2}\right).
\end{align*}
Moreover, a straightforward application of Lemma \ref{lem1} and the dominated convergence theorem shows that $\mbb{E}\left[N_{a,t}\right] = O(ah(a))$.  It follows that
\[
\mbb{P}\left(\left|\sum_{j=0}^{N_{a,t}-1} Z_{a,j}\right| > a\epsilon \right) \leq O\left(\frac{1}{\log a \cdot (\log\log a)^2}\right) \to 0
\]
as $a \to \infty$.

Next we consider the truncation by $I_1$.  As before we use the shorthand $Z_a := Z^{I_1}$.  To show that the sum $\sum_{j=0}^{N_{a,t}-1} a^{-1}Z_{a,j}$ converges to zero in probability, we use Chebyshev's inequality.  To this end, observe that as in the proof of Theorem  \ref{difflimitfiv}
\begin{align*}
\mbb{E}\left[\left(\sum_{j=0}^{N_{a,t}-1} a^{-1}Z_{a,j}\right)^2\right] &= \int_{-1}^1 \frac{1+\lambda}{1-\lambda}\frac{\mbb{E}\left[N_{a,t}+O(1)\right]}{ah(a)} \frac{h(a)}{a} \norm{Z_a}^2 \Pi_a(d\lambda),
\end{align*}
where we recall $\Pi_a = \norm{Z_a}^{-2}\left\langle Z_a, \Pi(d\lambda) Z_a \right\rangle$ is the spectral measure associated to $Z_a$.  Note $\norm{Z_a}^2 = O(\log\log a)$ and $h(a)/a = 1/\log a$, while all other factors in the integrand are bounded as $a \to \infty$.  The result follows.
\end{proof}

\begin{proof}[Proof of Lemma \ref{lem4}]
Because of the slight difference in definitions, we show separately that
$
\sum_{i=1}^{k_{a,t}-1} V_{a,i} \to 0$ and $ V_{a,k_{a,t}} \to 0.
$
To prove each of these, note that by Chebyshev's inequality, it suffices to show
\[
\mbb{E}\left[\left(\sum_{i=1}^{k_{a,t}-1} V_{a,i}\right)^2\right] \to 0, \quad \mbb{E}\left[V_{a,k_{a,t}}^2\right] \to 0,
\]
respectively.  We start with the first.  Observe that
\[
\mbb{E}\left[\left(\sum_{i=1}^{k_{a,t}-1} V_{a,i}\right)^2\right] = k_{a,t}\mbb{E}\left[V_{a,1}^2\right] + 2\sum_{1\leq i < j\leq k_{a,t}-1}\mbb{E}\left[V_{a,i}V_{a,j}\right]
\]
and
$$
\mbb{E}\left[V_{a,1}^2\right]
= \mbb{E}\left[\left(\sum_{j=0}^{s_{a,t}-1}a^{-1}Z_{a,j}\right)^2\right] 
= \int_\sigma \frac{1+\lambda}{1-\lambda} \frac{s_{a,t}+O(1)}{ah(a)} \frac{h(a)}{a} \norm{Z_a}^2 \Pi_a(d\lambda) 
= O\left((ah(a))^{\alpha-1}\right).
$$
Therefore,
$
k_{a,t}\mbb{E}\left[V_{a,1}^2\right] = O\left((ah(a))^{\alpha-\beta}\right)\to 0.
$
The number of terms in  $\sum_{1\leq i < j\leq r_{a,t}-1}V_{a,i}V_{a,j}$ is  $(k_{a,t}-1)(k_{a,t}-2)$  and each term $V_{a,i}V_{a,j}$ itself contains $s_{a,t}^2$ terms of the form $a^{-2}Z_{a,m}Z_{a,n}$.  It follows   that
\[
2\sum_{1\leq i < j\leq k_{a,t}-1}\mbb{E}\left[V_{a,i}V_{a,j}\right] = O((ah(a))^{2(1-\beta)+2\alpha-1})\to 0.
\]

Finally,   the sum in $V_{a,k_{a,t}}$ contains by definition less than $b_{a,t}+s_{a,t}$ terms.  Just as above, it follows that 
$
\mbb{E}\left[V_{a,k_{a,t}}^2\right] \leq O((ah(a))^{\beta-1}) \to 0.
$
\end{proof}

\begin{proof}[Proof of Lemma \ref{lem5}]
Let $w_i = \exp(i\mu U_{a,i})$ for $1 \leq i \leq k_{a,t}$.  Also let $\sigma^2_{w_i} = \mbb{E}\left[w_i^2\right] - \mbb{E}^2\left[w_i\right]$.  
Define $\mathcal{E}(w_1,w_2):=\mbb{E}\left[w_1w_2\right] - \mbb{E}\left[w_1\right]\mbb{E}\left[w_2\right].$
Because of the small block gap of size $s_{a,t}$ between big blocks,
$
|\mathcal{E}(w_1,w_2)| = |\sigma_{w_1}||\sigma_{w_2}||\corr(w_1,w_2)| 
\leq 4 |\corr(w_1,w_2)| 
\leq 4 |\rho(s_{a,t})|.
$
Moreover,
\begin{align*}
|\mbb{E}\left[w_1w_2w_3\right] - \mbb{E}\left[w_1\right]\mbb{E}\left[w_2\right]\mbb{E}\left[w_3\right]| &
\leq |\mathcal{E}(w_1, w_2w_3)|
+ |\mbb{E}\left[w_1\right]\mathcal{E}(w_2,w_3)| \\
&= |\sigma_{w_1}||\sigma_{w_2w_3}||\corr(w_1,w_2w_3)|  + |\mbb{E}[w_1]\mathcal{E}(w_2,w_3)| \\
&\leq 4|\rho(s_{a,t})| + 4|\rho(s_{a,t})|.
\end{align*}
It then follows by induction that
\[
\left|\mbb{E}\left[\exp\left(i\mu \sum_{i=1}^{k_{a,t}} U_{a,i}\right)\right) -\prod_{i=1}^{k_{a,t}}\mbb{E}\left[\exp\left(i\mu U_{a,i}\right)\right)\right| \leq 4k_{a,t}|\rho(s_{a,t})| \to 0
\]
as $a \to \infty$.
\end{proof}

\begin{proof}[Proof of Lemma \ref{lem6}]
A simple induction argument shows that if $z_1, \ldots z_m, w_1, \ldots, w_m \in \mbb{C}$ are of modulus at most 1, then
\[
\left|\prod_{i=1}^m z_i - \prod_{i=1}^m w_i\right| \leq \sum_{i=1}^m |z_i - w_i|.
\]
Applying the remark to $\left|\mbb{E}\left[\exp\left(i\mu U_{a,i}\right)\right)\right|$ and, for large enough $a$, $\left|1- \frac{\mu^2}{2k_{a,t}}t\mc{D}\right|$, we see
\begin{align*}
\left|\prod_{i=1}^{k_{a,t}}\mbb{E}\left[\exp\left(i\mu U_{a,i}\right)\right] - \left(1- \frac{\mu^2}{2k_{a,t}}t\mc{D}\right)^{k_{a,t}}\right| &\leq \sum_{i=1}^{k_{a,t}} \left|\mbb{E}\left[\exp\left(i\mu U_{a,i}\right)\right] - \left(1- \frac{\mu^2}{2k_{a,t}}t\mc{D}\right)\right| \\
&= k_{a,t} \left|\mbb{E}\left[\exp\left(i\mu U_{a,1}\right)\right] - \left(1- \frac{\mu^2}{2k_{a,t}}t\mc{D}\right)\right| \\
&= k_{a,t} \left|1 - \frac{\mu^2}{2}\mbb{E}\left[U_{a,1}^2\right] + O\left(\mbb{E}\left[U_{a,1}^3\right)\right] - \left(1- \frac{\mu^2}{2k_{a,t}}t\mc{D}\right)\right| \\
&\leq \frac{\mu^2}{2}\left|k_{a,t}\mbb{E}\left[U_{a,1}^2\right] - t\mc{D}\right| + \left|k_{a,t}O\left(\mbb{E}\left[U_{a,1}^3\right]\right)\right|.
\end{align*}
Now, it follows from the spectral representation of $\mc{D}$ that 
$
\left|k_{a,t}\mbb{E}\left[U_{a,1}^2\right] - t\mc{D}\right| \to 0.
$

Expanding $U_{a,1}^3$ results in $b_{a,t}^3$ terms of the form $a^{-3}Z_{a,i}Z_{a,j}Z_{a,k}$.  Suppose $i \leq j \leq k$ and let $D_1 = j-i$ and $D_2 = k-j$.  Let $C(a) = C\log a$, where $C > 4\beta$.  We separate the terms in the expansion of $U_{a,1}^3$ into one group containing  those terms whose indices satisfy $D_1 \leq C(a)$ and $D_2 \leq C(a)$; and another group with those terms
 for which $D_1 > C(a)$ or $D_2 > C(a)$.
Suppose $Z_{a,i}Z_{a,j}Z_{a,k}$ is in the first group.  Then  
$$
\mbb{E}\left[Z_{a,i}Z_{a,j}Z_{a,k}\right] \leq \mbb{E}^{1/3}\left[\left|Z_{a,i}\right|^3\right]\mbb{E}^{1/3}\left[\left|Z_{a,j}\right|^3\right]\mbb{E}^{1/3}\left[\left|Z_{a,k}\right|^3\right] = \mbb{E}\left[\left|Z_{a,i}\right|^3\right] = O\left(a/\log^{\gamma}a\right).
$$
Moreover,   of the $b_{a,t}^3$ total terms, $O\left(b_{a,t}C(a)^2\right)$ of them fall into this first group.  Thus  the contribution of these terms to $\left|k_{a,t}O\left(\mbb{E}\left[U_{a,1}^3\right]\right)\right|$ is at most of the order $O\left(\log^{1-\gamma}a\right)$ which is negligible as $a \to \infty$.

Suppose next that $Z_{a,i}Z_{a,j}Z_{a,k}$ is in the second group and assume without loss of generality that $D_1 > C(a)$.  Observe that
\[
\mbb{E}\left[Z_{a,i}Z_{a,j}Z_{a,k}\right] = \corr(Z_{a,i}, Z_{a,j}Z_{a,k})\,\mbb{E}^{1/2}\left[Z_{a,i}^2\right]\,\left(\mbb{E}\left[\left(Z_{a,j}Z_{a,k}\right)^2\right]-\mbb{E}^2\left[Z_{a,j}Z_{a,k}\right]\right)^{1/2}.
\]
Recall that $\mbb{E}\left[Z_{a,i}^2\right] = O\left(\log a\right)$.  Moreover, by the Cauchy-Schwarz inequality, stationarity, and further direct computation of moments
$$
\mbb{E}\left[\left(Z_{a,j}Z_{a,k}\right)^2\right] \leq \mbb{E}\left[Z_{a,i}^4\right] = O\left(a^2/\log^{2\gamma} a\right)
\text{ and }
\mbb{E}^2\left[Z_{a,j}Z_{a,k}\right] \leq \mbb{E}^2\left[Z_{a,j}^2\right] = O\left(\log^2a\right).
$$
Finally, using the mixing properties of the process
$
\corr\left(Z_{a,i}, Z_{a,j}Z_{a,k}\right) \leq \rho\left(C(a)\right) = O\left(a^{-C}\right).
$
Putting these three estimates together, we see that
\[
a^{-3}\mbb{E}\left[Z_{a,i}Z_{a,j}Z_{a,k}\right] \leq O\left(a^{-(C+2)}\log^{1/2-\gamma}a\right)
\]
Since, $O\left(b_{a,t}^3 - b_{a,t}C(a)^2\right) = O\left(b_{a,t}^3\right) = O\left(a^{6\beta}\log^{-3\beta}a\right)$ terms are in this second group,   their contribution  to $\left|k_{a,t}O\left(\mbb{E}\left[U_{a,1}^3\right]\right)\right|$ will   be of order at most $O\left(a^{4\beta-C}\log^{-1/2-2\beta-\gamma}a\right)$, which is indeed negligible with our choice of $C$.  The lemma now follows.
\end{proof}

\end{document}